\documentclass[onefignum,onethmnum,hidelinks]{siamart171218} 

\usepackage{amssymb}
\usepackage{mathtools}
\usepackage{graphicx}
\usepackage{epstopdf}
\usepackage{algorithmic}
\ifpdf
  \DeclareGraphicsExtensions{.eps,.pdf,.png,.jpg}
\else
  \DeclareGraphicsExtensions{.eps}
\fi

\newcommand{\TheTitle}{Minimax Separation of the Cauchy Kernel} 
\newcommand{\TheAuthors}{Jonathan E. Moussa}

\headers{\TheTitle}{\TheAuthors}

\ifpdf
\hypersetup{
  pdftitle={\TheTitle},
  pdfauthor={\TheAuthors}
}
\fi

\title{Minimax Separation of the Cauchy Kernel\thanks{Submitted to the editors September 14, 2019. \vspace{20 pt}
\funding{The Molecular Sciences Software Institute is supported by grant ACI-1547580 from the National Science Foundation.}}}

\author{Jonathan E. Moussa\thanks{Molecular Sciences Software Institute, Blacksburg, VA 24060
    (\email{godotalgorithm@gmail.com}).}}

\begin{document}

\maketitle

\begin{abstract}
We prove and apply an optimal low-rank approximation of the Cauchy kernel over separated real domains.
A skeleton decomposition is the minimum over real-valued functions of the maximum relative pointwise error.
We present an algorithm to optimize its parameters, demonstrate suboptimal but effective heuristic approximations, and identify numerically stable forms.
\end{abstract}


\begin{AMS}
 15A03, 15B05, 32A26, 49K35 \vspace{24 pt}
\end{AMS}

\section{Introduction}

Low-rank approximations of both matrices \cite{lowrank_matrix_survey} and bivariate functions \cite{chebfun2d} are useful primitives in numerical analysis.
For example, they are used in hierarchical matrices \cite{hierarchical_matrix_survey}
 and low-rank approximations of tensors and multivariate functions \cite{lowrank_tensor_survey}.
Truncated singular value decompositions (SVDs) are popular low-rank approximations
 because they are simple to compute and optimal with respect to the 2-norm.
For a matrix $\mathbf{K} \in \mathbb{R}^{m \times n}$
 or an integral kernel $K: \mathcal{X} \times \mathcal{Y} \rightarrow \mathbb{R}$
 between two Lebesgue-integrable function spaces $L^2(\mathcal{X})$ and $L^2(\mathcal{Y})$,
 we can build a minimizer of
 \begin{equation} \label{svd_optimal}
 \min_{\substack{\mathbf{F} \in \mathbb{R}^{r \times m} \\ \mathbf{G} \in \mathbb{R}^{r \times n}}} \left\| \mathbf{K} - \mathbf{F}^T \mathbf{G} \right\|_2 
 \ \ \ \ \mathrm{or} \ \ \ \
  \min_{\substack{ \mathbf{f} \in L^2(\mathcal{X})^r \\ \mathbf{g} \in L^2(\mathcal{Y})^r}} \left\| K(x,y) -  \mathbf{f}(x)^T \mathbf{g}(y) \right\|_2
\end{equation}
 by retaining the $r$ largest singular values and vectors in the SVD of $\mathbf{K}$ or $K(x,y)$.

In this paper, we present a new optimal low-rank approximation result with both conceptual and practical value.
We summarize this result in the following theorem.
\begin{theorem} \label{minimax_theorem}
The maximum relative pointwise error in rank-$r$ approximations of $1/(x - y)$
 minimized over sets of real-valued functions $\mathcal{F}(\mathcal{X})$ and $\mathcal{F}(\mathcal{Y})$
 on compact real domains $\mathcal{X}$ and $\mathcal{Y}$ such that $\max \mathcal{X} > \min \mathcal{X} > \max \mathcal{Y} > \min \mathcal{Y}$ reduces to
\begin{equation} \label{minimax}
 \min_{\substack{\mathbf{f} \in \mathcal{F}(\mathcal{X})^r \\ \mathbf{g} \in \mathcal{F}(\mathcal{Y})^r}} \max_{\substack{x \in \mathcal{X} \\ y \in \mathcal{Y}}} \left| 1 - (x-y) \mathbf{f}(x)^T \mathbf{g}(y) \right|
  = Z_r(\mathcal{X},\mathcal{Y})
\end{equation}
  for the Zolotarev number $Z_r(\mathcal{X},\mathcal{Y})$, which can be defined on such domains as
\begin{equation} \label{simple_minimax}
 Z_r(\mathcal{X},\mathcal{Y}) \coloneqq \min_{\substack{ \tilde{\mathbf{x}} \in [\min \mathcal{X},\max \mathcal{X}]^r \\ \tilde{\mathbf{y}} \in [\min \mathcal{Y},\max \mathcal{Y} ]^r}}
 \max_{\substack{x \in \mathcal{X} \\ y \in \mathcal{Y}}} \left| \prod_{i=1}^r \frac{(x - \tilde{x}_i)(y - \tilde{y}_i)}{(x - \tilde{y}_i)(y - \tilde{x}_i)} \right|.
\end{equation}
Minimizers of \cref{minimax} and \cref{simple_minimax} are related by a skeleton decomposition,
\begin{equation} \label{skeleton}
 \mathbf{f}(x)^T \mathbf{g}(y) = \mathbf{C}(x,\tilde{\mathbf{y}}) \mathbf{C}(\tilde{\mathbf{x}},\tilde{\mathbf{y}})^{-1} \mathbf{C}(\tilde{\mathbf{x}}, y), \ \ \ \ 
 [\mathbf{C}(\mathbf{x},\mathbf{y})]_{i,j} \coloneqq \frac{1}{x_i - y_j}.
\end{equation}
\end{theorem}
\cref{minimax_theorem} is an example of an integral kernel and error metric for which a skeleton decomposition is optimal rather than a truncated SVD.
Extensions of this result to other kernels or error metrics are likely to be limited because of the specificity of its proof.
However, a skeleton decomposition may remain superior to a truncated SVD in similar circumstances,
 thus this specific exact result may beget a more diverse set of useful approximations.
Also, the approximation power of skeleton decompositions and truncated SVDs are related,
 thus an improved understanding of one can benefit the other.
For example, the Zolotarev number in \cref{minimax} is also a part of upper bounds
 on the minimum 2-norm error attainable by truncated SVDs of matrices with a low displacement rank \cite{displacement.rank.review},
 and singular values are part of upper bounds on the minimum 2-norm error attainable by skeleton decompositions of general matrices \cite{skeleton}.

The paper proceeds as follows.
In \cref{proof}, we prove \cref{minimax_theorem}.
In \cref{many_solutions}, we review the known analytical solutions to \cref{simple_minimax} when $\mathcal{X}$ and $\mathcal{Y}$ are closed intervals,
 compute numerical solutions when $\mathcal{X}$ and $\mathcal{Y}$ are finite unions of closed intervals,
 and construct heuristic solutions when $\mathcal{X}$ and $\mathcal{Y}$ have finite cardinality.
In \cref{svd_comparison}, we compare \cref{skeleton} with the truncated SVD to motivate several numerically stable forms for \cref{skeleton} and 
 analyze their ability to approximate each other based on the equivalence between the norms that they minimize.
In \cref{conclusions}, we conclude with a summary of possible future extensions and applications of \cref{minimax_theorem}.

\section{Main proof\label{proof}}

Our strategy for proving \cref{minimax_theorem} is to show that $Z_r(\mathcal{X}, \mathcal{Y})$ is both an upper and a lower bound in \cref{minimax}.
This upper bound has been proved for both the Cauchy kernel \cite[p. 332]{displacement.rank.review} and the closely related Hilbert kernel \cite[p. 429]{Hilbert_upper_bound}
 by relating their low-rank approximations to separable relative error functions,
\begin{equation} \label{residual_reduction}
 \forall h \in \mathcal{R}_{r,r}, \ \ \ \ \exists \mathbf{f}, \mathbf{g} \in \mathcal{R}_{r,r}^r \ \ \ \ \mathrm{s. t.} \ \ \ \ \frac{1}{x - y} \left( 1 - \frac{h(x)}{h(y)} \right) = \mathbf{f}(x)^T \mathbf{g}(y),
\end{equation}
 for sets of rational functions, $\mathcal{R}_{m,n} \coloneqq \{ p/q : p \in \mathcal{P}_m, q \in \mathcal{P}_n \}$, where $\mathcal{P}_n$ is the set of polynomials of degree at most $n$.
The upper bound holds when $\mathcal{X}$ and $\mathcal{Y}$ are closed disjoint subsets of the extended complex plane,
 while the lower bound requires that they be real, compact, and separated.
To simplify the presentation of the proof, we start with a Lemma to reconcile our nonstandard definition of $Z_r(\mathcal{X},\mathcal{Y})$ in \cref{simple_minimax}
 and prepare for the construction of numerical solutions of $Z_r(\mathcal{X},\mathcal{Y})$ in \cref{many_solutions}.

\begin{lemma} \label{Z_lemma}
The standard definition of $(n,n)$th Zolotarev numbers \cite[eq. (1.1)]{zolotarev_number_def},
\begin{equation} \label{Z_def}
 Z_{n}(\mathcal{X},\mathcal{Y}) \coloneqq \inf_{h \in \mathcal{R}_{n,n}} \frac{\sup_{x \in \mathcal{X}} \left| h(x) \right|}{\inf_{y \in \mathcal{Y}}|h(y)|}
\end{equation}
where $\mathcal{X}$ and $\mathcal{Y}$ are closed disjoint subsets of the extended complex plane,
 is equivalent to \cref{simple_minimax} when $\mathcal{X}$ and $\mathcal{Y}$ are real and compact and $\max \mathcal{X} > \min \mathcal{X} > \max \mathcal{Y} > \min \mathcal{Y}$.
If $\min\{|\mathcal{X}|,|\mathcal{Y}|\} > n$ also, then it is strictly monotonic, $Z_{n+1}(\mathcal{X},\mathcal{Y}) < Z_{n}(\mathcal{X},\mathcal{Y})$, and
 has a unique minimizer up to a nonzero multiplicative constant characterized by
\begin{equation} \label{Z_equioscillation}
 \frac{h(x_i)}{h(y_j)} = (-1)^{i+j} \frac{\max_{x \in \mathcal{X}} \left| h(x) \right|}{\min_{y \in \mathcal{Y}}|h(y)|}, \ \ \ x_i < x_{i+1}, \ \ \ y_j > y_{j+1}, \ \ \ 1 \le i,j \le n+1,
\end{equation}
 for $\mathbf{x} \in \mathcal{X}^{n+1}$ and $\mathbf{y} \in \mathcal{Y}^{n+1}$ that respectively interleave minimizers $\tilde{\mathbf{x}}$ and $\tilde{\mathbf{y}}$ of \cref{simple_minimax}.
\end{lemma}

\begin{proof}
For $\min\{|\mathcal{X}|,|\mathcal{Y}|\} \le n$, we can cover $\mathcal{X}$ or $\mathcal{Y}$ with the roots or poles of $h$ respectively to attain a trivial minimum, $Z_n(\mathcal{X},\mathcal{Y}) = 0$.
Therefore we only consider the relevant nontrivial case of $\min\{|\mathcal{X}|,|\mathcal{Y}|\} > n$ with a real compact $\mathcal{X}$ and $\mathcal{Y}$ such that
 $\max \mathcal{X} > \min \mathcal{X} > \max \mathcal{Y} > \min \mathcal{Y}$.
 Here, $Z_n(\mathcal{X},\mathcal{Y}) \neq 0$ because a nonzero $h(z)$ must have finite nonzero values at any $z \in \mathcal{X} \cup \mathcal{Y}$
  that does not correspond to a root or pole, and there are too few roots and poles to cover $\mathcal{X}$ or $\mathcal{Y}$ for any $h \in \mathcal{R}_{n,n}$.

First, we establish that \cref{Z_equioscillation} is necessary for minimizers of \cref{Z_def}.
Without loss of generality, we use polynomials $p$ and $q$ as minimization variables such that $h = p/q$
 and restrict their roots such that the outer minimand is well-defined and attained,
\begin{equation} \label{poly_representation}
 Z_n(\mathcal{X}, \mathcal{Y}) = \inf_{\substack{p \in \mathcal{P}_n^{\mathcal{Y}} \\ q \in \mathcal{P}_n^{\mathcal{X}}}} \max_{\substack{x \in \mathcal{X} \\ y \in \mathcal{Y}}} \left| \frac{p(x) q(y)}{q(x) p(y)} \right|,
  \ \ \ \ \mathcal{P}^{\mathcal{Z}}_n \coloneqq \{ f \in \mathcal{P}_n : f(z) \neq 0, \forall z \in \mathcal{Z} \}.
\end{equation}
For a given $q \in \mathcal{P}_n^{\mathcal{X}}$ and $y \in \mathcal{Y}$, we study the set of minimizing $p$.
The same analysis applies to the set of minimizing $q$ for a given $p \in \mathcal{P}_n^{\mathcal{Y}}$ and $x \in \mathcal{X}$.
We replace $p$ with $\hat{p} = p/p(y)$ as the minimization variable to isolate $y$ as a domain constraint,
\begin{equation}
  \min_{p \in \mathcal{P}_n^{\mathcal{Y}}} \max_{x \in \mathcal{X}} \left| \frac{p(x) q(y)}{q(x) p(y)} \right| =
  |q(y)| \min_{\substack{\hat{p} \in \mathcal{P}_n^{\mathcal{Y}} \\ \hat{p}(y) = 1}} \max_{x \in \mathcal{X}} \left| \frac{\hat{p}(x)}{q(x)} \right| . \notag
\end{equation}
If we represent $\hat{p}$ as $\hat{p}(x) = c (x^n - b(x))$ for $c \in \mathbb{C}$ and $b \in \mathcal{P}_{n-1}$, then we can relate the set of minimizers
 to a weighted polynomial approximation problem \cite[Chap.\@ 3]{approximation.theory},
\begin{equation}
 \underset{\substack{\hat{p} \in \mathcal{P}_n^{\mathcal{Y}} \\ \hat{p}(y) = 1}}{\mathrm{arg \, min}} \max_{x \in \mathcal{X}} \left| \frac{\hat{p}(x)}{q(x)} \right| =
  \left\{ \frac{x^n - \tilde{b}(x)}{y^n - \tilde{b}(y)} : \tilde{b} \in \underset{b \in \mathcal{P}_{n-1}}{\mathrm{arg \, min}} \max_{x \in \mathcal{X}} \left| \frac{x^n - b(x)}{q(x)} \right| \right\} . \notag  
\end{equation}
Such $\tilde{b}$ are unique and attain the maximum at $n+1$ ordered points in $\mathcal{X}$ where the sign of $(x^n - b(x))/|q(x)|$ alternates.
The unique minimizing $\hat{p}$ is indeed in $\mathcal{P}_n^{\mathcal{Y}}$ since it has $n$ simple roots in $[\min \mathcal{X}, \max \mathcal{X}]$.
The corresponding set of minimizing $p$,
\begin{equation}
 \underset{p \in \mathcal{P}_n^{\mathcal{Y}}}{\mathrm{arg \, min}} \max_{x \in \mathcal{X}} \left| \frac{p(x) q(y)}{q(x) p(y)} \right| =
  \left\{ c(x^n - \tilde{b}(x)) : c \in \mathbb{C} \setminus \{0\}, \tilde{b} \in \underset{b \in \mathcal{P}_{n-1}}{\mathrm{arg \, min}} \max_{x \in \mathcal{X}} \left| \frac{x^n - b(x)}{q(x)} \right| \right\}, \notag
\end{equation}
 does not depend on $y$ and thus is also the set of minimizers for \cref{poly_representation} at fixed $q \in \mathcal{P}_n^{\mathcal{X}}$.
Therefore it is necessary for $p$ and $q$ minimizers to satisfy equioscillation constraints,
\begin{align} \label{polynomial_necessary}
 \frac{p(x_i)}{|q(x_i)|} &= s (-1)^i \max_{x \in \mathcal{X}} \left|\frac{p(x)}{q(x)}\right| , \ \ \ \ \ \, x_i < x_{i+1}, \ \ \ \ 1 \le i \le n, \ \ \ \ \mathrm{and} \\
 \frac{q(y_j)}{|p(y_j)|} &= s' (-1)^j \max_{y \in \mathcal{Y}} \left| \frac{q(y)}{p(y)} \right| , \ \ \ \ \, y_j > y_{j+1}, \ \ \ \ 1 \le j \le n, \notag
\end{align}
 for some $s, s' \in \{-1,1\}$, $\mathbf{x} \in \mathcal{X}^{n+1}$, and $\mathbf{y} \in \mathcal{Y}^{n+1}$.
With their constrained roots, we can remove the modulus from $|q(x_i)|$ and $|p(y_j)|$ in \cref{polynomial_necessary} and set $s = s'$.
Necessity of \cref{Z_equioscillation} follows from substituting $h = p/q$ into \cref{polynomial_necessary} and multiplying constraints.
 
Next, we establish that minimizers of \cref{simple_minimax} and \cref{Z_def} exist and are equivalent to each other.
Since we have already established that minimizers of \cref{Z_def} must have $n$ simple roots in $[\min \mathcal{X}, \max \mathcal{X}]$
 and $n$ simple poles in $[\min \mathcal{Y}, \max \mathcal{Y}]$ when $\mathcal{X}$ and $\mathcal{Y}$ are real, compact, and separated,
 we can restrict the minimization domain to rational functions that satisfy these constraints without excluding minimizers.
We transform from \cref{Z_def} to \cref{simple_minimax} by representing $h$ as a product of roots $\tilde{\mathbf{x}}$ and poles $\tilde{\mathbf{y}}$ on restricted domains.
Because the optimand of \cref{simple_minimax} is continuous and bounded over its compact domains,
 minimizers exist by Berge's maximum theorem \cite[Chap.\@ 6, sect. 3]{maximum_theorem}.

After that, we establish the sufficiency of \cref{Z_equioscillation} for minimizers of \cref{Z_def} and their uniqueness up to a nonzero multiplicative constant.
Our proof is directly inspired by established mappings between minimizers of \cref{Z_def} and optimal rational approximants \cite[Thm. 2.1]{zolotarev34}
 and between minimizers of \cref{Z_def} for which $n$ differs by a factor of two \cite[sect. 2]{Zolotarev_derivation}.
For a given minimizer $h \in \mathcal{R}_{n,n}$ of \cref{Z_def}, we rescale it to satisfy
\begin{equation} \label{h_normalization}
 \sup_{x \in \mathcal{X}} |h(x)| = \sqrt{Z_n(\mathcal{X},\mathcal{Y})}, \ \ \ \  \inf_{y \in \mathcal{Y}} |h(y)| = 1 / \sqrt{Z_n(\mathcal{X},\mathcal{Y})},
\end{equation}
 and be real-valued on $\mathbb{R}$.
We consider an invertible map between $h^2$ and $k$ in $\mathcal{R}_{2n,2n}$, 
\begin{equation} \label{rational_approximant}
 k(z) = (1+Z_n(\mathcal{X},\mathcal{Y})) \frac{1 - h(z)^2}{1 + h(z)^2}, \ \ \ \ h(z)^2 = \frac{1+Z_n(\mathcal{X},\mathcal{Y}) - k(z)}{1+Z_n(\mathcal{X},\mathcal{Y}) + k(z)},
\end{equation}
 where $k$ is a prospective minimizer of a weighted rational approximation problem,
\begin{align} \label{rational_approximation_problem}
 & \ \ \ \min_{k \in \mathcal{R}_{2n,2n}} \max_{z \in [\min \mathcal{Y},\max \mathcal{X}]} \omega(z) \left| \theta(z) - k(z) \right| = Z_n(\mathcal{X},\mathcal{Y}), \\
  \theta(z) & \coloneqq \left\{ \begin{array}{ll} 1, & z \ge \min \mathcal{X} \\ -1 , & z < \min \mathcal{X} \end{array} \right. , \ \ \ \
  \omega(z) \coloneqq \left\{ \begin{array}{ll} 1, & z \in \mathcal{X} \cup \mathcal{Y} \\ Z_n(\mathcal{X},\mathcal{Y}) / 3, & z \not\in \mathcal{X} \cup \mathcal{Y} \end{array} \right. \notag .
\end{align}
In generalized rational approximation theory \cite[Chap.\@ 5]{approximation.theory},
 rational functions over an interval domain with a positive weight function have Haar-subspace structure,
 which guarantees a unique minimizing $k$ characterized by attaining the maximum error at $m$ ordered points in $[\min \mathcal{Y},\max \mathcal{X}]$
 where the sign of $\omega(z) ( \theta(z) - k(z) )$ alternates.
If $k \not\in \mathcal{R}_{2n,2n-1} \cup \mathcal{R}_{2n-1,2n}$, then $m = 4 n + 2$.
Since \cref{rational_approximant} connects the range of $h$ and $k$
 from $(-\infty, -1/\sqrt{Z_n(\mathcal{X},\mathcal{Y})}] \cup [ 1/\sqrt{Z_n(\mathcal{X},\mathcal{Y})}, \infty)$
 to $[-1 - Z_n(\mathcal{X},\mathcal{Y}), -1 + Z_n(\mathcal{X},\mathcal{Y})]$ and from $[-\sqrt{Z_n(\mathcal{X},\mathcal{Y})}, \sqrt{Z_n(\mathcal{X},\mathcal{Y})}]$ to
 $[1 - Z_n(\mathcal{X},\mathcal{Y}), 1 + Z_n(\mathcal{X},\mathcal{Y})]$,
 the $n$ roots of $h$ and $n + 1$ local minima of $|h(y)|$ for $y \in \mathcal{Y}$
 correspond to the $2n+1$ local minima of $\omega(z) ( \theta(z) - k(z) )$ for $z \in \mathcal{X} \cup \mathcal{Y}$
 and the $n$ poles of $h$ and $n+1$ local maxima of $|h(x)|$ for $x \in \mathcal{X}$ correspond to its $2n+1$ local maxima.
Because the unweighted error is bounded by the triangle inequality and $Z_n(\mathcal{X},\mathcal{Y}) \le 1$,
 \begin{equation}
  \left| \theta(z) - k(z) \right| \le 3, \ \ \ \ z \in \mathbb{R}, \notag
\end{equation}
 the weighted error cannot have a larger maximum in $[\min \mathcal{Y},\max \mathcal{X}] \setminus (\mathcal{X} \cup \mathcal{Y})$.
Thus, a unique minimizing $k$ of \cref{rational_approximation_problem} is mapped by \cref{rational_approximant}
 to a unique minimizing $h$ of \cref{Z_def} up to a nonzero multiplicative constant.
Similarly, the sufficiency of error alternation for a minimizing $k$
 corresponds to the sufficiency of \cref{Z_equioscillation} for a minimizing $h$ because
 \cref{rational_approximant} maps the local extrema of $\omega(z) ( \theta(z) - k(z) )$ for $z \in \mathcal{X} \cup \mathcal{Y}$ to the interleaved
 roots and local maxima of $|h(x)|$ for $x \in \mathcal{X}$ and $1/|h(y)|$ for $y \in \mathcal{Y}$.

Finally, we establish the strict monotonicity of $Z_n(\mathcal{X}, \mathcal{Y})$.
It is monotonic because $\mathcal{R}_{n,n} \subset \mathcal{R}_{n+1,n+1}$ and $Z_{n+1}(\mathcal{X},\mathcal{Y})$ is minimized over a larger domain than $Z_n(\mathcal{X}, \mathcal{Y})$.
For $\min\{|\mathcal{X}|,|\mathcal{Y}|\} > n+1$,
  monotonicity is strict because $Z_n(\mathcal{X}, \mathcal{Y})$ and $Z_{n+1}(\mathcal{X},\mathcal{Y})$ have unique minimizers with different numbers of roots and poles,
  and their equality would contradict this uniqueness.
For $\min\{|\mathcal{X}|,|\mathcal{Y}|\} = n+1$, monotonicity is strict because $Z_n(\mathcal{X}, \mathcal{Y}) \neq 0$ and $Z_{n+1}(\mathcal{X}, \mathcal{Y}) = 0$.
\end{proof}

The purpose of \cref{Z_lemma} in the proof of \cref{minimax_theorem} is to simplify its end point.
Instead of proving a direct equivalence between \cref{minimax} and \cref{simple_minimax}, we are only required to prove the equivalence between \cref{minimax} and \cref{Z_def}
 before invoking \cref{Z_lemma}.

\begin{proof}[Proof of \cref{minimax_theorem}]
We focus on the nontrivial case of $\min \{ |\mathcal{X}|, |\mathcal{Y}| \} > r$, since
 $Z_r(\mathcal{X}, \mathcal{Y})=0$ can be achieved by \cref{skeleton} for $\min \{ |\mathcal{X}|, |\mathcal{Y}| \} \le r$ by covering $\mathcal{X}$ or $\mathcal{Y}$
  with elements of $\tilde{\mathbf{x}}$ or $\tilde{\mathbf{y}}$ respectively.
We refer to the left-hand side of \cref{minimax} as $\tilde{Z}_r(\mathcal{X},\mathcal{Y})$ in this proof,
 thus our goal is to show that $\tilde{Z}_r(\mathcal{X},\mathcal{Y}) = Z_r(\mathcal{X},\mathcal{Y})$.
We can readily show that $\tilde{Z}_r(\mathcal{X},\mathcal{Y}) \le Z_r(\mathcal{X},\mathcal{Y})$ by restricting the minimization domain of $\mathbf{f}$ and $\mathbf{g}$ in \cref{minimax}
 to \cref{residual_reduction} and replacing $\mathbf{f}$ and $\mathbf{g}$ in the optimand with $h$.
The root-pole representation of $h$ in \cref{simple_minimax} then corresponds to $\mathbf{f}(x)^T\mathbf{g}(y)$ in \cref{skeleton}.
The rest of the proof is focused on showing that $\tilde{Z}_r(\mathcal{X},\mathcal{Y}) \ge Z_r(\mathcal{X},\mathcal{Y})$ using a sequence of relaxations.

The primary form of relaxation is the max-min inequality,
\begin{equation}
 \inf_{a \in \mathcal{A}} \sup_{b \in \mathcal{B}} f(a,b) \ge \sup_{b \in \mathcal{B}} \inf_{a \in \mathcal{A}} f(a,b), \notag
\end{equation}
 for any function $f: \mathcal{A} \times \mathcal{B} \rightarrow \mathbb{R}$.
We split the maximization over $\mathcal{Y}$ in $\tilde{Z}_r(\mathcal{X},\mathcal{Y})$
 into maximizations over subsets of $\mathcal{Y}$ and their elements
 and use the max-min inequality,
\begin{align} \label{maxmin_relax}
 \tilde{Z}_r(\mathcal{X},\mathcal{Y}) &= \inf_{\substack{\mathbf{f} \in \mathcal{F}(\mathcal{X})^r \\ \mathbf{g} \in \mathcal{F}(\mathcal{Y})^r}} 
  \sup_{\substack{x \in \mathcal{X} \\ \{y_1, \cdots , y_{r+1} \} \subseteq \mathcal{Y}}} \max_{i} \left| 1 - (x-y_i) \mathbf{f}(x)^T \mathbf{g}(y_i) \right| \\
 &\ge
  \inf_{\mathbf{g} \in \mathcal{F}(\mathcal{Y})^r}
  \sup_{\substack{x \in \mathcal{X} \\ \{y_1, \cdots , y_{r+1} \} \subseteq \mathcal{Y}}}
  \min_{\mathbf{f} \in \mathbb{R}^r} \max_{i} \left| 1 - (x-y_i) \mathbf{f}^T \mathbf{g}(y_i) \right| , \notag
\end{align}
 which results in the independent minimization of $\mathbf{f}(x) \in \mathbb{R}^r$ at each $x \in \mathcal{X}$.

The inner minimax problem in \cref{maxmin_relax} is equivalent to a linear program,
\begin{align}
 \tilde{h}(x) &\coloneqq \min_{\mathbf{f} \in \mathbb{R}^r} \max_{i} \left| 1 - (x-y_i) \mathbf{f}^T \mathbf{g}(y_i) \right| 
 = \min_{\tilde{\mathbf{f}} \in \mathbb{R}^n} \max_{i} \left| 1 - (x-\tilde{y}_i) \tilde{\mathbf{f}}^T \tilde{\mathbf{g}}(\tilde{y}_i) \right| \notag \\
 &= \min \{ a \in \mathbb{R} : \ -a \le 1 - (x - \tilde{y}_i) \tilde{\mathbf{g}}(\tilde{y}_i)^T \tilde{\mathbf{f}} \le a, \ 1 \le i \le n+1, \ \tilde{\mathbf{f}} \in \mathbb{R}^n \}, \notag
\end{align}
  for some $\{ \tilde{g}_1, \cdots , \tilde{g}_n \} \subseteq \{ g_1, \cdots , g_r \}$ that are linearly independent when their domain is restricted
  to some $\{ \tilde{y}_1, \cdots , \tilde{y}_{n+1} \} \subseteq \{ y_1, \cdots , y_{r+1} \}$.
There is a minimizing $\tilde{\mathbf{f}}$ and $a$ that saturates one inequality per pair, which is one of the possible solutions of
\begin{equation}
\left[ \begin{array}{cc} (x - \tilde{y}_1) \tilde{\mathbf{g}}(\tilde{y}_1)^T & s_1 \\ \vdots & \vdots \\ (x - \tilde{y}_{n+1}) \tilde{\mathbf{g}}(\tilde{y}_{n+1})^T & s_{n+1} \end{array} \right]
 \left[ \begin{array}{c} \tilde{\mathbf{f}} \\ a \end{array} \right] = 
\left[ \begin{array}{c} 1 \\ \vdots \\ 1 \end{array} \right] , \ \ \ \ \mathbf{s} \in \{ -1, 1\}^{n+1}. \notag
\end{equation}
We solve for $a$ by using Cramer's rule and cofactor expansions into cofactors $\tilde{c}_i$
 and calculate $\tilde{h}(x)$ by minimizing over $\mathbf{s}$ to maximize the denominator,
\begin{align} \label{inner_solution}
 \tilde{h}(x) &= \min_{\mathbf{s} \in \{-1,1\}^{n+1}} \left| \frac{\sum_{i=1}^{n+1} \frac{\tilde{c}_i}{x - \tilde{y}_i}}{\sum_{i=1}^{n+1} \frac{s_i \tilde{c}_i}{x - \tilde{y}_i}} \right|
  = \frac{\left|\sum_{i=1}^{n+1} \frac{\tilde{c}_i}{x - \tilde{y}_i}\right|}{\sum_{i=1}^{n+1} \left| \frac{\tilde{c}_i}{x - \tilde{y}_i} \right|}, \\
  \tilde{c}_i &\coloneqq (-1)^i \det [ \ \tilde{\mathbf{g}}(\tilde{y}_1) \ \cdots \ \tilde{\mathbf{g}}(\tilde{y}_{i-1}) \ \tilde{\mathbf{g}}(\tilde{y}_{i+1}) \ \cdots \ \tilde{\mathbf{g}}(\tilde{y}_{n+1}) \ ], \notag
\end{align}
 which is well defined for $x \not\in \mathcal{Y}$ because there is at least one nonzero $\tilde{c}_i$ value.

The next relaxation follows from the systematically improvable approximation of $\tilde{h}(x)$ on $x \in \mathcal{X}$ with a maximum error $\delta>0$
 satisfying $\tilde{h}_{\delta}(y_i) = 1$ for $1 \le i \le r+1$,
\begin{equation}
 \tilde{h}_{\delta}(z) \coloneqq \frac{\left|\sum_{i=1}^{r+1} \frac{c_i}{z - y_i}\right|}{\sum_{i=1}^{r+1} \left| \frac{c_i}{z - y_i} \right|},
 \ \ \ \ c_i \coloneqq \left\{ \begin{array}{ll} 2 r \tilde{c}_j / \sum_{k=1}^{n+1} |\tilde{c}_k |, & \exists j \ \mathrm{s.t.} \  y_i = \tilde{y}_j \ \mathrm{and} \ \tilde{c}_j \neq 0 \\ \left( \frac{\min \mathcal{X} - \max \mathcal{Y}}{\max \mathcal{X} - \min \mathcal{Y}}\right)^2 \delta, & \mathrm{otherwise} \end{array} \right. . \notag
\end{equation}
We construct a lower bound for $\tilde{h}(x)$ using the triangle inequality and insert a trivial maximization of $1/\tilde{h}_{\delta}(y_i)$ over $1 \le i \le r+1$,
\begin{equation}
 \tilde{h}(x) \ge \tilde{h}_{\delta}(x) - \left| \tilde{h}_{\delta}(x) - \tilde{h}(x) \right| \ge
  \max_{i} \frac{\tilde{h}_{\delta}(x)}{\tilde{h}_{\delta}(y_i)} - \delta ,
 \ \ \ \ x \in \mathcal{X}. \notag
\end{equation}
We then replace the inner minimax problem in \cref{maxmin_relax} with its solution $\tilde{h}(x)$ in \cref{inner_solution} as the new optimand
 of the outer minimax problem and relax it as
\begin{equation} \label{penultimate_bound}
 \tilde{Z}_r(\mathcal{X},\mathcal{Y}) \ge \inf_{\mathbf{g} \in \mathcal{F}(\mathcal{Y})^r}
  \sup_{\substack{x \in \mathcal{X} \\ \{y_1, \cdots , y_{r+1}\} \subseteq \mathcal{Y}}}
  \max_i \frac{\tilde{h}_{\delta}(x)}{\tilde{h}_{\delta}(y_i)} - \delta ,
\end{equation}
 which is a valid lower bound for any $\delta > 0$.

The final relaxation expands and simplifies the minimization domain following a decoupling of minimization and maximization variables with the general form
\begin{equation} \label{decoupling_trick}
  \inf_{a \in \mathcal{A}} \sup_{c \in \mathcal{C}} f( b(a,c), c) = \inf_{b' \in \mathcal{B}} \sup_{c \in \mathcal{C}} f(b',c)
\end{equation}
 for any pair of functions, $f: \mathcal{B} \times \mathcal{C} \rightarrow \mathbb{R}$ and $b: \mathcal{A} \times \mathcal{C} \rightarrow \mathcal{B}$,
 such that for any $b' \in \mathcal{B}$ and $c \in \mathcal{C}$ there exists $a \in \mathcal{A}$ satisfying $b' = b(a,c)$.
For any $\tilde{h}_{\delta}(z)$, we can represent any $h_{\delta} \in \mathcal{R}_{r,r}$ in barycentric form \cite{barycentric_interpolation}
 using some $\mathbf{p},\mathbf{q} \in \mathbb{R}^{r+1}$ as
\begin{equation}
 h_{\delta}(z) \coloneqq \left( \sum_{i=1}^{r+1} \frac{p_i c_i}{z - y_i} \middle) \middle/ \middle( \sum_{i=1}^{r+1} \frac{q_i |c_i|}{z - y_i} \right) \notag
\end{equation}
  such that $\tilde{h}_\delta(x) = |h_{\delta}(x)|$ for $x \in \mathcal{X}$ when $p_i = q_i = 1$ for $1 \le i \le r+1$.
We insert $\mathbf{p}$ and $\mathbf{q}$ as minimization variables to relax \cref{penultimate_bound},
 apply \cref{decoupling_trick} to replace $\mathbf{g}$, $\mathbf{p}$, and $\mathbf{q}$ minimizations by $h \in \mathcal{R}_{r,r}$,
 and regroup the $\mathbf{y}$ and $i$ maximizations back to $y \in \mathcal{Y}$,
\begin{align}
 \tilde{Z}_r(\mathcal{X},\mathcal{Y}) &\ge \inf_{\substack{\mathbf{g} \in \mathcal{F}(\mathcal{Y})^r \\ \mathbf{p},\mathbf{q} \in \mathbb{R}^{r+1} }}
  \sup_{\substack{x \in \mathcal{X} \\ \{y_1, \cdots , y_{r+1}\} \subseteq \mathcal{Y}}}
  \max_i \left| \frac{h_{\delta}(x)}{h_{\delta}(y_i)} \right| - \delta \notag \\
   &= \inf_{h \in \mathcal{R}_{r,r}}
  \sup_{\substack{x \in \mathcal{X} \\ \{y_1, \cdots , y_{r+1}\} \subseteq \mathcal{Y}}}
  \max_i \left| \frac{h(x)}{h(y_i)} \right| - \delta
  = \inf_{h \in \mathcal{R}_{r,r}}
  \sup_{\substack{x \in \mathcal{X} \\ y \in \mathcal{Y}}} \left| \frac{h(x)}{h(y)} \right| - \delta . \notag
\end{align}
In the $\delta \rightarrow 0$ limit, this lower bound becomes $Z_r(\mathcal{X},\mathcal{Y})$ in \cref{simple_minimax} by \cref{Z_lemma}.
\end{proof}

With the proof concluded, it is worthwhile to highlight the details that constrain the $\mathcal{X}$ and $\mathcal{Y}$ domains in \cref{minimax_theorem} and \cref{Z_lemma}.
Realness and compactness enable the equioscillation of polynomial minimizers in \cref{poly_representation},
 thus constraining their roots to be simple and in a prescribed interval.
It is plausible that complex $p$ and $q$ minimizers of \cref{poly_representation} have roots respectively confined to convex hulls of $\mathcal{X}$ and $\mathcal{Y}$,
 but this is not straightforward to show.
Separation enables a corresponding separation between the minimization and maximization domains in \cref{simple_minimax}
 to guarantee a bounded continuous optimand without excluding possible minimizers.
Realness and separation enable the representation of $\tilde{h}(x)$ in \cref{inner_solution} as the modulus of a rational function for $x \in \mathcal{X}$ since
 $|x - y| = x - y$ for $y \in \mathcal{Y}$.
The inner minimax problem in \cref{maxmin_relax} still can be solved in the disjoint complex case by minimizing $\mathbf{s}$ in \cref{inner_solution} over $\{ \phi \in \mathbb{C} : |\phi| = 1 \}$ rather than $\{-1,1\}$,
 but the resulting lower bound is not $Z_r(\mathcal{X}, \mathcal{Y})$ and may be unattainable.

\section{Solutions of \cref{simple_minimax}\label{many_solutions}}

While \cref{minimax_theorem} relates minimizers of \cref{minimax} and \cref{simple_minimax},
 it does not provide specific solutions to the optimization problem in \cref{simple_minimax} or ways to construct them.
Here we discuss some analytical, numerical, and heuristic solutions.
First, we review the analytical solutions of $Z_n(\lambda) \coloneqq Z_n( [\lambda , 1], [ -1 , -\lambda] )$ for $\lambda \in (0,1)$
 corresponding to Zolotarev's third problem \cite{zolotarev.review,minimax.rational}.
Next, we prescribe an iterative algorithm that converges quadratically to numerical solutions of $Z_n(\mathcal{X},\mathcal{Y})$ for $\mathcal{X}$ and $\mathcal{Y}$
 that are finite unions of closed intervals.
Optimality of these solutions is certified by their characterization in \cref{Z_lemma}.
Finally, we construct heuristic solutions for $\mathcal{X}$ and $\mathcal{Y}$ of finite cardinality using the analytical solutions of $Z_n(\lambda)$
 and compare them to numerical solutions over a simple statistical distribution of $\mathcal{X}$ and $\mathcal{Y}$.

About the uniqueness of solutions, the minimizers of \cref{simple_minimax} and \cref{Z_def} are unique up to a choice of ordering and normalization if $Z_n(\mathcal{X}, \mathcal{Y}) \neq 0$,
 but minimizers of \cref{minimax} are not unique.
We have prescribed a convenient normalization of $h$ in \cref{h_normalization},
 and a similarly convenient ordering of $\tilde{\mathbf{x}}$ and $\tilde{\mathbf{y}}$ that is compatible with \cref{Z_equioscillation} is
\begin{equation}
 x_i < \tilde{x}_i < x_{i+1} , \ \ \ \ y_i > \tilde{y}_i > y_{i+1} , \ \ \ \ 1 \le i \le n.
\end{equation}
When the minimizing $\tilde{\mathbf{x}}$ and $\tilde{\mathbf{y}}$ of \cref{simple_minimax} are unique,
 there is also a unique $\mathbf{f}(x)^T \mathbf{g}(y)$ of the form in \cref{skeleton} that minimizes \cref{minimax}.
However, the minimized maximum in \cref{minimax} is only attained at $(r+1)^2$ pairs of $x$ and $y$ values defined by \cref{Z_equioscillation}.
We can alter $\mathbf{f}(x)$ and $\mathbf{g}(y)$ for any $x$ or $y$ not contained in this subset of points without changing the maximum,
 thus unique minimizers of \cref{minimax} require extra constraints such as \cref{skeleton}.

\subsection{Analytical solutions\label{analytical_solution}}

Zolotarev solved four problems in polynomial and rational approximation using elliptic functions \cite{zolotarev.review},
 and his third problem was $Z_n(\lambda)$.
Its original reference \cite{minimax.rational} has no English translation, but a review of the solution is available in English \cite[Chap.\@ 9]{elliptic_functions}.
We use an independent solution from the appendix of \cite{Zolotarev_derivation},
 in which the roots $\tilde{\mathbf{x}}$, poles $\tilde{\mathbf{y}}$, extrema $\mathbf{x}$ in $\mathcal{X}$, and extrema $\mathbf{y}$ in $\mathcal{Y}$
 of $h$ in \cref{Z_def} are uniformly spaced in a mapped domain
 defined by a Jacobi elliptic function,
\begin{align} \label{zolotarev_solution}
  \tilde{x}_i &= -\tilde{y}_i = \xi\left( \frac{i - 1/2}{n}\right), \ \ \ \ 1 \le i \le n, \\
   x_j  &= -y_j = \xi\left( \frac{j - 1}{n} \right), \ \ \ \ 1 \le j \le n+1, \notag \\
 \xi(v) &\coloneqq \mathrm{dn}\left( (1 - v) K\left(\sqrt{1 - \lambda^2}\right), \sqrt{1 - \lambda^2}\right), \notag
\end{align}
 for the delta amplitude $\mathrm{dn}(u,k)$ with quarter period $K(k)$ as specified by the elliptic modulus $k$ \cite[Chap.\@ 22]{special_functions}.
These functions have to be evaluated carefully when $\lambda$ is small enough for $\sqrt{1 - \lambda^2}$ to be rounded to $1$.
For example, $K$ can be evaluated using the imaginary quarter period $K'$ as $K\left(\sqrt{1 - \lambda^2}\right) = K'(\lambda)$,
 and $\xi(v)$ can be evaluated recursively by the ascending Landen transformation \cite[eq. (22.7.8)]{special_functions} for small $\lambda$.

Because $\mathcal{R}_{n,n}$ is invariant to M\"{o}bius transformations of the domain, we can use a M\"{o}bius transformation to map \cref{zolotarev_solution}
 to the solution of \cref{Z_def} for $\mathcal{X} = [x_{\min} , x_{\max}]$ and $\mathcal{Y} = [y_{\min} , y_{\max}]$ such that $x_{\min} > y_{\max}$
 if $\lambda$ is chosen to match cross-ratios,
\begin{align} \label{mobius_transform}
  \lambda = \frac{\left( \sqrt{|(x_{\max}-x_{\min})(y_{\max}-y_{\min})|} - \sqrt{|(x_{\min}-y_{\min})(x_{\max}-y_{\max})|} \right)^2}{|(x_{\min} - y_{\max})(x_{\max} - y_{\min})|} , \\
    z \mapsto - \frac{ (1-\lambda) (\lambda+z) x_{\min} x_{\max} + (1+\lambda) (\lambda - z) y_{\max} x_{\max} + 2 \lambda (z-1) y_{\max} x_{\min} }
 { (1-\lambda) (\lambda + z) y_{\max} + (1+\lambda) (\lambda - z) x_{\min}  + 2 \lambda (z - 1) x_{\max}} . \notag
\end{align}
This domain mapping does not change the value of $Z_n(\mathcal{X},\mathcal{Y})$.
It also works for any $\mathcal{Y} = [ y_{\min},\infty) \cup (-\infty,y_{\max}] $ that additionally satisfies $y_{\min} > x_{\max}$.
\cref{minimax_theorem} and \cref{Z_lemma} can be extended to a closed $\mathcal{Y} = \mathcal{Y}_+ \cup \mathcal{Y}_-$ satisfying
 $\min \mathcal{Y}_+ > \max \mathcal{X}$ and $\min \mathcal{X} > \max \mathcal{Y}_-$ by incorporating this M\"{o}bius transformation into their proofs.

Several limits and bounds are useful for applying and analyzing these solutions.
Tight lower and upper bounds on $Z_n(\lambda)$ are known for large $n$ \cite[Cor. 3.2]{displacement.rank.review},
\begin{align} \label{Z_bounds}
 \frac{4 \rho^{-2n}}{(1 + \rho^{-4n})^4} &\le Z_n(\lambda) = \prod_{i=1}^n \left( \frac{1 - \tilde{x}_i }{1 + \tilde{x}_i } \right)^2 \le \frac{4 \rho^{-2n}}{(1 + \rho^{-4n})^2} \le 4 \rho^{-2n} \le 4 \tilde{\rho}^{-2n}, \\
  \rho &\coloneqq \exp\left( \pi \frac{K(\lambda)}{K'(\lambda)} \right), \ \ \ \ \tilde{\rho} \coloneqq \exp\left(\frac{\pi^2/2}{\log(4/\lambda)} \right), \notag
\end{align}
 and $\tilde{\rho}$ approaches $\rho$ in the limit of small $\lambda$ \cite[eq.\@ (19.9.5)]{special_functions}.
The $\xi(v)$ map function reduces to elementary special functions as $\lambda$ approaches its limiting values of $0$ and $1$,
\begin{equation} \label{map_limits}
 \lim_{\lambda \rightarrow 0} \frac{\log(\xi(v))}{\log(\lambda)} = 1 - v, \ \ \ \  \lim_{\lambda \rightarrow 1} \frac{1 - \xi(v)}{1-\lambda} = \frac{1 + \cos(\pi v)}{2},
\end{equation}
 from limits \cite[Table 22.5.4]{special_functions} and series expansions \cite[eq.\@ (22.11.3)]{special_functions} of $\mathrm{dn}(u,k)$.

\subsection{Numerical solutions\label{numerical_solution}}

We represent minimizers of \cref{simple_minimax} and \cref{Z_def} as
\begin{equation}
 h(z) = \exp(b) \prod_{i=1}^n \frac{z - \tilde{x}_i}{z - \tilde{y}_i}, \ \ \ \ \tilde{x}_j < \tilde{x}_{j+1},\ \ \ \ \tilde{y}_j > \tilde{y}_{j+1}, \ \ \ \ 1 \le j \le n-1, \notag
\end{equation}
 for $b \in \mathbb{R}$ and label local extrema between roots and poles of the $\mathcal{X}$ and $\mathcal{Y}$ domains as
\begin{equation} \label{local_maximizers}
 x_i \in \underset{x \in \mathcal{X} \cap [\tilde{x}_i , \tilde{x}_{i+1}]}{\mathrm{arg \, max}} |h(x)|, \ \ \ \
 y_i \in \underset{y \in \mathcal{Y} \cap [\tilde{y}_{i+1} , \tilde{y}_{i}]}{\mathrm{arg \, max}} |h(y)|^{-1}, \ \ \ \ 
 1 \le i \le n+1.
\end{equation}
Consistent with the characterization of minimizers in \cref{Z_lemma},
 we only consider $\tilde{\mathbf{x}}$ and $\tilde{\mathbf{y}}$ for which these maximization domains are not empty.
We use the logarithm of the equioscillation constraints in \cref{Z_equioscillation} to characterize numerical solutions by
\begin{equation} \label{log_characterization}
 a = \log | h(x_i) | , \ \ \ \  a = - \log | h(y_i) | , \ \ \ \ 1 \le i \le n+1,
\end{equation}
 for an unknown equioscillation magnitude $a$.
Starting from an initial trial minimizer,
 we iterative refine its variables $a$, $b$, $\tilde{\mathbf{x}}$, and $\tilde{\mathbf{y}}$ until \cref{log_characterization} is satisfied.

Since \cref{log_characterization} is nonlinear in $\tilde{\mathbf{x}}$ and $\tilde{\mathbf{y}}$,
 we linearize the equations in these variables to calculate first-order corrections $\delta \tilde{\mathbf{x}}$ and $\delta \tilde{\mathbf{y}}$.
They are defined by the linear system
\begin{align} \label{linear_search}
 & \ \ \ \, \left[ \begin{array}{cccc} \mathbf{1} & -\mathbf{1} & \mathbf{C}(\mathbf{x},\tilde{\mathbf{x}}) & -\mathbf{C}(\mathbf{x},\tilde{\mathbf{y}}) \\
 \mathbf{1}  & \mathbf{1} & -\mathbf{C}(\mathbf{y},\tilde{\mathbf{x}}) & \mathbf{C}(\mathbf{y},\tilde{\mathbf{y}}) \end{array} \right]
 \left[ \begin{array}{c} a \\ b \\ \delta \tilde{\mathbf{x}} \\ \delta \tilde{\mathbf{y}} \end{array} \right] =
 \left[ \begin{array}{c} \mathbf{c} \\ \mathbf{d} \end{array} \right] , \\
  c_i &= \sum_{j=1}^n \log \left| \frac{x_i - \tilde{x}_j}{x_i - \tilde{y}_j} \right| , \ \ \ \ d_i = \sum_{j=1}^n \log \left| \frac{y_i - \tilde{y}_j}{y_i - \tilde{x}_j} \right| , \ \ \ \ 1 \le i \le n+1 \notag ,
\end{align}
 where $\mathbf{1}$ is a vector with all elements equal to one.
A $(2n+2)$-by-$(2n+1)$ submatrix of this matrix equation is a diagonally-weighted Cauchy matrix, which facilitates an analytical solution.
We can solve it using Cramer's rule, cofactor expansions, and the Cauchy determinant formula \cite[eq.\@ (4)]{Cauchy_formulas}, which results in
\begin{align} \label{search_formula}
 a &= \frac{\sum_i (c_i p_i + d_i q_i)}{\sum_i (p_i + q_i)}, \ \ \ \ p_i = \frac{ \prod_j (x_i - \tilde{y}_j) (x_i - \tilde{x}_j) }{\prod_j (x_i - y_j) \prod_{j \neq i} (x_i - x_j)} , \\
 b = \sum_{i} (x_i (a - c_i & ) p_i + y_i (a - d_i) q_i), \ \ \ \ q_i = \frac{ \prod_j (\tilde{x}_j - y_i) (\tilde{y}_j - y_i) }{\prod_{j} (x_j - y_i) \prod_{j \neq i} (y_j - y_i)} , \notag \\
 \delta \tilde{x}_i &= \frac{ \prod_j (y_j - \tilde{x}_i) (x_j - \tilde{x}_i) }{\prod_j (\tilde{y}_j - \tilde{x}_i) \prod_{j \neq i} (\tilde{x}_j - \tilde{x}_i)} \sum_j \left( \frac{a - c_j}{\tilde{x}_i - x_j}p_j + \frac{a - d_j}{\tilde{x}_i - y_j}q_j \right) , \notag \\
 \delta \tilde{y}_i &= \frac{ \prod_j (\tilde{y}_i - x_j) (\tilde{y}_i - y_j) }{\prod_j (\tilde{y}_i - \tilde{x}_j) \prod_{j \neq i} (\tilde{y}_i - \tilde{y}_j)} \sum_j \left( \frac{a - c_j}{\tilde{y}_i - x_j}p_j + \frac{a - d_j}{\tilde{y}_i - y_j}q_j \right) . \notag
\end{align}
The equioscillation conditions are satisfied when the elements of $\delta \tilde{\mathbf{x}}$ and $\delta \tilde{\mathbf{y}}$ are zero,
 which corresponds to equal-element right-hand side vectors, $\mathbf{c} = c_1 \mathbf{1}$ and $\mathbf{d} = d_1 \mathbf{1}$.

While the linearization of $\tilde{\mathbf{x}}$ and $\tilde{\mathbf{y}}$ is convenient for defining \cref{linear_search},
 nonlinearities are strong in these variables and can stagnate an iterative solution process.
With the search direction for updated solution variables $\tilde{\mathbf{x}}'$ and $\tilde{\mathbf{y}}'$ defined by
\begin{equation} \label{old_search}
 \tilde{\mathbf{x}}' = \tilde{\mathbf{x}} + \alpha \, \delta \tilde{\mathbf{x}}, \ \ \ \ \tilde{\mathbf{y}}' = \tilde{\mathbf{y}} + \alpha \, \delta \tilde{\mathbf{y}}, \ \ \ \
 \alpha \in [0,1],
\end{equation}
 the largest deviation from satisfying \cref{log_characterization} can always be reduced for sufficiently small $\alpha$,
 but the total amount of reduction per linear solution update might be small.
We find nonlinearities to be substantially weaker in variables $\mathbf{s}$ and $\mathbf{t}$ defined by
\begin{equation}
 s_i = \log \left( \frac{\tilde{x}_i - x_i}{x_{i+1} - \tilde{x}_i} \right), \ \ \ \ t_i = \log \left( \frac{\tilde{y}_i - y_{i+1}}{y_{i} - \tilde{y}_i} \right), \ \ \ \ 1 \le i \le n. \notag
\end{equation}
The first-order corrections in these two sets of variables are linearly related,
 and we can define the related search direction for updated variables $\mathbf{s}'$ and $\mathbf{t}'$ as
\begin{align} \label{new_search}
  \mathbf{s}' = \mathbf{s} + \alpha \, \delta \mathbf{s}, \ \ \ \ \mathbf{t}' &= \mathbf{t} + \alpha \, \delta \mathbf{t}, \ \ \ \
 \alpha \in [0,1], \\
 \delta s_i = \frac{\delta \tilde{x}_i (x_{i+1}-x_i)}{(x_{i+1} - \tilde{x}_i)(\tilde{x}_i - x_i)}, \ \ \ \
 \delta t_i &= \frac{\delta \tilde{y}_i (y_i - y_{i+1})}{(y_i - \tilde{y}_i)(\tilde{y}_i - y_{i+1})}, \ \ \ \
 1 \le i \le n. \notag
\end{align}
In practice, we observe that \cref{new_search} reduces the largest deviation from satisfying \cref{log_characterization}
 for larger $\alpha$ values than \cref{old_search}, producing a larger overall reduction.

We implement\footnote{An ANSI C implementation is available in the supplementary materials and will be maintained on GitHub at https://github.com/godotalgorithm/zolotarev-number.}
 a simple algorithm that is greater than $99 \%$ reliable in practice.
Since $Z_n(\mathcal{X},\mathcal{Y})$ is invariant to M\"{o}bius transformations,
 we transform sets to satisfy
\begin{equation}
 \min \mathcal{Y} = -1, \ \ \ \ \max \mathcal{Y} = -\lambda, \ \ \ \ \min \mathcal{X} = \lambda, \ \ \ \ \max \mathcal{X} = 1, \notag
\end{equation}
 for some $\lambda \in (0,1)$ to improve numerical behavior.
The initial values of $\tilde{\mathbf{x}}$ and $\tilde{\mathbf{y}}$ are constructed by inserting $n+1$ points each into $\xi^{-1}(\mathcal{X})$ and $\xi^{-1}(-\mathcal{Y})$ for $\xi$ in \cref{zolotarev_solution},
 starting at zero and inserting new points as far as possible from previously inserted points,
 ordering them into vectors $\mathbf{a}$ and $\mathbf{b}$, and assigning $\tilde{x}_i = \xi(\frac{1}{2}(a_i + a_{i+1}))$ and $\tilde{y}_i = -\xi(\frac{1}{2}(b_i + b_{i+1}))$.
The main iterative loop alternates between calculating local extrema $\mathbf{x}$ and $\mathbf{y}$ in \cref{local_maximizers},
 calculating corrections $\delta \tilde{\mathbf{x}}$ and $\delta \tilde{\mathbf{y}}$ in \cref{search_formula},
 and performing a search over $\alpha$ in \cref{new_search} to update $\tilde{\mathbf{x}}$ and $\tilde{\mathbf{y}}$.
We choose $\alpha$ to minimize the difference between the largest and smallest values of $\log | h(x_i) |$ and $-\log | h(y_i)|$ in \cref{log_characterization} using a golden section search
 and terminate the loop when this quantity can no longer be decreased.
It converges quadratically in an asymptotic regime of small $\|\delta \tilde{\mathbf{x}}\|$ and $\|\delta \tilde{\mathbf{y}}\|$ before stagnating at its numerical floor.
This algorithm requires $O(n)$ memory, and each update of $\mathbf{x}$, $\mathbf{y}$, $\delta \tilde{\mathbf{x}}$, $\delta \tilde{\mathbf{y}}$, $\tilde{\mathbf{x}}$, or $\tilde{\mathbf{y}}$
 in each iteration requires $O(n^2)$ operations.
We restrict the implementation to $\mathcal{X}$ and $\mathcal{Y}$ that are both finite unions of closed intervals
 to simplify the test of set inclusion for $\mathcal{X}$ and $\mathcal{Y}$ to a binary search.

Our simple algorithm and implementation have some theoretical and numerical limitations that could be improved with further development effort.
We do not have a rigorous explanation for the effectiveness of the search direction in \cref{new_search} over the straightforward choice in \cref{old_search}.
While a continuous infinitesimal update of $\tilde{\mathbf{x}}$ and $\tilde{\mathbf{y}}$ by $\delta \tilde{\mathbf{x}}$ and $\delta \tilde{\mathbf{y}}$ in \cref{search_formula}
 with continuous updates of the local extrema $\mathbf{x}$ and $\mathbf{y}$ in \cref{local_maximizers}
 monotonically reduces the minimand of $Z_n(\mathcal{X}, \mathcal{Y})$ until it is minimized,
 our iterative algorithm has convergence behavior that is not so straightforward to analyze.
Even though rapid convergence occurs in most test cases, there are infrequent pathological cases that either stagnate or fail to converge.
We also observe numerical problems when a root or pole of a minimizing $h$
 approaches a local extremum at an isolated point in $\mathcal{X}$ or $\mathcal{Y}$ with
 a distance that exponentially decreases in $n$.
The approximate floating-point representations of these numbers become identical
 even though their exact difference is nonzero and can be approximated accurately with a floating-point number.
This numerical problem can be repaired by representing roots and poles as differences from the nearest local extremum, which complicates the implementation.
Also, it remains possible to compute $\log Z_n(\mathcal{X}, \mathcal{Y})$ and $\log | h(z) |$ when $Z_n(\mathcal{X}, \mathcal{Y})$ and
 $h(z)$ underflow their approximate floating-point representations by carefully avoiding
 intermediate quantities that can underflow and computing the logarithm of products that can underflow as sums of their individual logarithms.
Such an implementation would sacrifice some performance in exchange for reliability because logarithms are more computationally expensive than multiplication and division.

\subsection{Heuristic solutions\label{heuristic_solutions}}

For some applications of \cref{minimax_theorem} and $Z_n(\mathcal{X}, \mathcal{Y})$,
 a heuristic solution can be as useful as the exact minimizer
 if upper bounds such as in \cref{Z_bounds} are used instead of the exact value of $Z_n(\mathcal{X}, \mathcal{Y})$
 and the minimand of \cref{simple_minimax} still satisfies these bounds with the heuristic solution.
The simplest example of this is using the minimizer of $Z_n([\min \mathcal{X}, \max \mathcal{X}], [\min \mathcal{Y}, \max \mathcal{Y}])$
 from \cref{analytical_solution} as a heuristic solution of $Z_n(\mathcal{X}, \mathcal{Y})$.
However, a single outlying element of either $\mathcal{X}$ or $\mathcal{Y}$ can degrade the accuracy of this heuristic solution.
Here we generalize this heuristic solution to be more accurate when $\mathcal{X}$ and $\mathcal{Y}$ both have finite cardinality.

To construct heuristic solutions, we label the set elements as $\mathcal{X} = \{ x_1, \cdots , x_{|\mathcal{X}|}\}$ and
 $\mathcal{Y} = \{ y_1, \cdots , y_{|\mathcal{Y}|}\}$ ordered such that $x_{i-1} < x_{i}$ for $2 \le i \le |\mathcal{X}|$ and $y_{i-1} > y_{i}$ for $2 \le i \le |\mathcal{Y}|$.
We then partition these sets using non-negative integers $n_-$ and $n_+$ as
\begin{align}
 \mathcal{X} = \mathcal{X}_- \cup \mathcal{X}_0 \cup \mathcal{X}_+ , \ \ \ \
 \mathcal{X}_- &\coloneqq \{x_1 , \cdots , x_{n_-} \} , \notag \\
\mathcal{X}_0 &\coloneqq \{x_{n_- + 1} , \cdots , x_{|\mathcal{X}| - n_+} \}, \notag \\
\mathcal{X}_+ &\coloneqq \{x_{|\mathcal{X}| - n_+ + 1} , \cdots , x_{|\mathcal{X}|} \} , \notag \\
 \mathcal{Y} = \mathcal{Y}_- \cup \mathcal{Y}_0 \cup \mathcal{Y}_+ , \ \ \ \ 
\mathcal{Y}_- &\coloneqq \{y_1 , \cdots , y_{n_-} \} , \notag \\
\mathcal{Y}_0 &\coloneqq \{y_{n_- + 1} , \cdots , y_{|\mathcal{Y}| - n_+} \}, \notag \\
\mathcal{Y}_+ &\coloneqq \{y_{|\mathcal{Y}| - n_+ + 1} , \cdots , y_{|\mathcal{Y}|} \} , \notag
\end{align}
 which must satisfy $n_- + n_+ \le n$.
The elements of $\tilde{\mathbf{x}}$ and $\tilde{\mathbf{y}}$ are heuristically chosen to reduce the objective function in \cref{simple_minimax} on different parts of its domain.
We reduce it to zero if either $x \in \mathcal{X}_- \cup \mathcal{X}_+$ or $y \in \mathcal{Y}_- \cup \mathcal{Y}_+$ by choosing $n_- + n_+$ elements
 to cover $\mathcal{X}_- \cup \mathcal{X}_+$ and $\mathcal{Y}_- \cup \mathcal{Y}_+$ respectively.
We set the remaining $n - n_- - n_+$ elements to the analytical solution in \cref{analytical_solution} defined by
 $x_{\min} = x_{n_- + 1}$, $x_{\max} = x_{|\mathcal{X}| - n_+}$, $y_{\min} = y_{|\mathcal{Y}| - n_+}$, and $y_{\max} = y_{n_- + 1}$,
 to bound from above their contributions to the objective function by $Z_{n - n_- - n_+}(\lambda)$ for $\lambda$ in \cref{mobius_transform}.
The remaining contributions when $x \in \mathcal{X}_0$ and $y \in \mathcal{Y}_0$ can be grouped into cross-ratios and independently maximized
 to construct an upper bound on $Z_n(\mathcal{X}, \mathcal{Y})$ that is satisfied by this heuristic solution,
\begin{align} \label{heuristic_upper_bound}
 Z_n(\mathcal{X}, \mathcal{Y}) & \le Z_{n - n_- - n_+}(\lambda) \prod_{i=1}^{n_-} \frac{(x_{|\mathcal{X}| - n_+} - x_{i}) (y_{i} - y_{|\mathcal{Y}| - n_+})}{(x_{|\mathcal{X}| - n_+} - y_{i}) (x_{i} - y_{|\mathcal{Y}| - n_+})} \\
 & \ \ \ \times \prod_{i=0}^{n_+ - 1} \frac{(x_{|\mathcal{X}| - i} - x_{n_- + 1}) (y_{n_- + 1} - y_{|\mathcal{Y}| - i})}{(x_{|\mathcal{X}| - i} - y_{n_- + 1})(x_{n_- + 1} - y_{|\mathcal{Y}| - i})} \le Z_{n - n_- - n_+}(\lambda) \notag .
\end{align}
The simplest heuristic solution and its upper bound are recovered for  $n_- = n_+ = 0$.
Computing these upper bounds do not require any optimization steps,
 but the most accurate heuristic solutions are obtained by minimizing a bound over $n_-$ and $n_+$.

\begin{figure}[t]
  \centering
  \label{fig_heuristic}\includegraphics{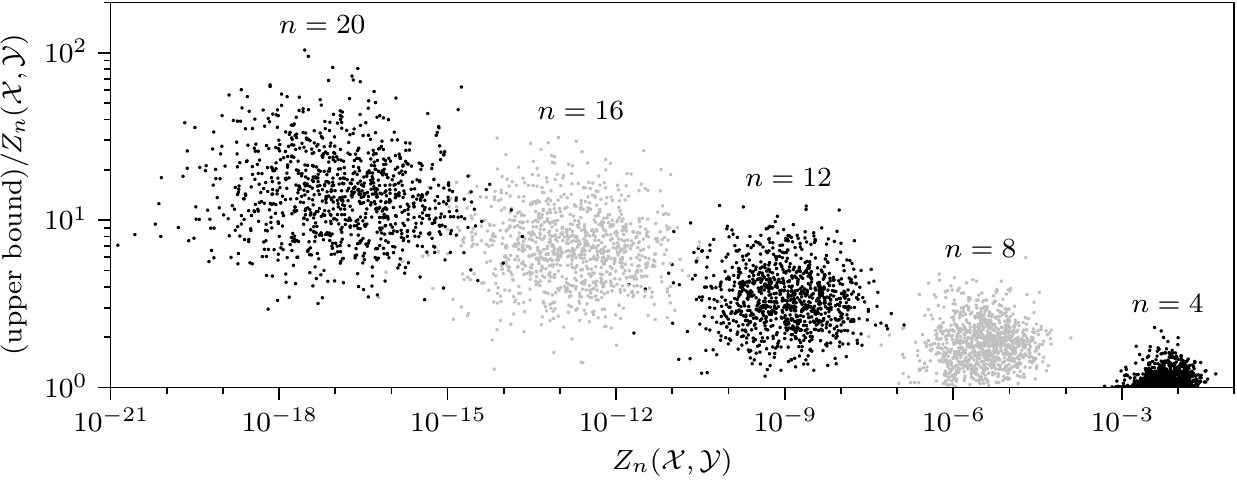}
  \caption{Tightness of the first upper bound in \cref{heuristic_upper_bound} for 1000 samples drawn from a distribution of $\mathcal{X}$ and $\mathcal{Y}$
 that each contain 100 points sampled from uniform distributions over $(0,1]$ and $[-1,0)$ respectively.
 The color of data points is alternated between neighboring $n$ values for contrast.}
\end{figure}

To examine the practical value of these heuristic solutions, we compare them to numerical solutions in \cref{fig_heuristic}.
The upper bounds of $Z_n(\mathcal{X}, \mathcal{Y})$ in \cref{heuristic_upper_bound} loosen with increasing $n$,
 which corresponds to an increasing suboptimality of heuristic solutions.
However, the overall exponential decay, $Z_n \approx 0.12^{n-2}$, greatly outpaces the growing inefficiency, $\tilde{Z}_n / Z_n \approx 1.2^{n - 4}$,
 in the upper bound $\tilde{Z}_n$ that is satisfied by the heuristic solution.
Asymptotically, a fractional increase in $n$ enables the heuristic solution to match the decay of the optimal solution, $\tilde{Z}_{\lceil 1.1 n \rceil} \approx Z_n$.
Thus, heuristic solutions are nearly as effective as numerical solutions in this example, especially for small $n$.

Heuristic solutions can be further extended to $\mathcal{X}$ and $\mathcal{Y}$ that are finite unions of closed intervals
 by choosing outlying subsets in $\mathcal{X}$ or $\mathcal{Y}$ and using analytical solutions for covering intervals of these subsets to bound the objective function. 
However, we will need greedy strategies for partitioning to avoid a high-dimensional combinatorial optimization of partitioning parameters.
With only two parameters, $n_+$ and $n_-$, it is inexpensive to minimize the upper bounds in \cref{heuristic_upper_bound} exhaustively,
 but this strategy is not efficiently scalable to larger numbers of parameters.

\section{Comparison between skeleton decompositions and truncated SVDs\label{svd_comparison}}

For a skeleton decomposition to be as useful in practice as a truncated SVD, it ought
 to retain their beneficial numerical properties and be of comparable flexibility as an approximant.
A matrix $\mathbf{A}$ factored into its SVD, $\mathbf{A} = \mathbf{U} \mathbf{D} \mathbf{V}^T$,
 is numerically stable to reconstruct by multiplying $\mathbf{U}$, $\mathbf{D}$, and $\mathbf{V}^T$
 since $\mathbf{U}$ and $\mathbf{V}$ are orthogonal matrices that do not amplify rounding errors in floating-point arithmetic.
This is not the case for a skeleton decomposition that is grouped into a product of Cauchy matrices and their inverses as in \cref{skeleton},
 whereby an ill-conditioned intermediate matrix can amplify rounding errors during matrix multiplication.
Regarding flexibility, a truncated SVD minimizes \cref{svd_optimal} and a skeleton decomposition minimizes \cref{minimax},
 but their effectiveness as approximate minimizers of the error metrics for which they are suboptimal is not obvious.
The optimal error sets a lower bound on their error, and the equivalence of norms sets an upper bound.
Their actual errors can be anywhere in between,
 and it is possible for one approximant to be more transferrable between error metrics.

In this section, we demonstrate the numerical stability and flexibility of skeleton decompositions using numerical examples and limited theoretical analysis.
Efficiently computable condition numbers are defined for the purpose of quantifying numerical stability,
 and they are observed to be small in practice.
The coefficients that govern norm equivalence are then derived,
 and a common exponential decay is observed for both error metrics and both approximants.
However, the prefactors of this common exponential decay have substantial variations between metrics and approximants.

\subsection{Stable forms\label{stable_form}}

Our numerical analysis of matrix decompositions utilizes a generic upper bound for errors in floating-point summation \cite[eq.\@ (2.6)]{floating_point_summation},
\begin{equation} \label{sum_error}
 \left| fl\left( \sum_{i=1}^n a_i \right) - \sum_{i=1}^n a_i \right| \le \epsilon_n \sum_{i=1}^n | a_i | ,
\end{equation}
 where $fl(x)$ refers to the unspecified evaluation of $x$ in floating-point arithmetic and
 $\epsilon_n \approx (n-1) \epsilon_2$ is the specific error associated with the summation of $n$ floating-point numbers.
It is often useful to relax the absolute sum on the right-hand side of \cref{sum_error} into a weaker but more convenient expression.
For example, the elementwise error in reconstructing a matrix $\mathbf{A} \in \mathbb{R}^{m \times n}$ from its SVD, $\mathbf{A} = \mathbf{U} \mathbf{D} \mathbf{V}^T$, can be relaxed to
\begin{equation} \label{svd_rounding}
 \left| fl\left([\mathbf{U} \mathbf{D} \mathbf{V}^T]_{i,j}\right) - [\mathbf{A}]_{i,j} \right| \le \epsilon_{\min\{m,n\}} \| \mathbf{A} \|_2, \ \ \ \ 1 \le i \le m,  \ \ \ \ 1 \le j \le n,
\end{equation}
through bounding the elements of $\mathbf{D}$ by their maximum value $\|\mathbf{A}\|_2$
  and the sum over columns of $\mathbf{U}$ and $\mathbf{V}$ by one using the Cauchy--Schwarz inequality.
We seek to modify \cref{skeleton} into one or more stable forms and establish an error bound similar to \cref{svd_rounding}.

We construct three numerically stable forms for \cref{skeleton} by regrouping its matrices into
 one-sided and two-sided interpolative matrix decompositions \cite{interpolative_decomposition3},
\begin{align} \label{interpolant_matrices}
 &\mathbf{f}(x)^T \mathbf{g}(y) = \mathbf{u}(x)^T \mathbf{C}(\tilde{\mathbf{x}},y) = \mathbf{C}(x,\tilde{\mathbf{y}}) \mathbf{v}(y) = \mathbf{u}(x)^T \mathbf{C}(\tilde{\mathbf{x}},\tilde{\mathbf{y}}) \mathbf{v}(y), \\
 & \ \ \ \ \mathbf{u}(x) \coloneqq \mathbf{C}(\tilde{\mathbf{y}},\tilde{\mathbf{x}})^{-1} \mathbf{C}(\tilde{\mathbf{y}},x) , \ \ \ \ \mathbf{v}(y) \coloneqq \mathbf{C}(\tilde{\mathbf{x}},\tilde{\mathbf{y}})^{-1} \mathbf{C}(\tilde{\mathbf{x}},y). \notag
\end{align}
The interpolation vectors $\mathbf{u}(x)$ and $\mathbf{v}(y)$ combine the Lagrange polynomials
 from the Cauchy matrix inverse formula \cite[eq.\@ (7)]{Cauchy_formulas} with rational weight functions in $\mathcal{R}_{0,r}$,
\begin{equation}
 u_i(x) = L_i(x, \tilde{\mathbf{x}}) \prod_{j=1}^r \frac{\tilde{x}_i - \tilde{y}_j}{x - \tilde{y}_j}, \ \ \ \ 
 v_i(y) = L_i(y, \tilde{\mathbf{y}}) \prod_{j=1}^r \frac{\tilde{x}_j - \tilde{y}_i}{\tilde{x}_j - y}, \notag
\end{equation}
 using the notation $L_i(x,\mathbf{y}) \coloneqq \prod_{j\neq i} (x - y_j)/(y_i - y_j)$ for Lagrange polynomials.
Both $\mathbf{u}(x)$ and $\mathbf{v}(y)$ retain the interpolation property of their Lagrange polynomials,
\begin{equation}
 u_i(\tilde{x}_j) = v_i(\tilde{y}_j) = \left\{ \begin{array}{ll} 1 , & i = j \\ 0 , & i \neq j \end{array} \right. , \ \ \ \
 1 \le i,j \le r, \notag
\end{equation}
 and also their normalization property,
\begin{equation}
 \sum_{i=1}^r u_i(z) =  \sum_{i=1}^r v_i(z) = 1. \notag
\end{equation}
Unfortunately, $u_i(x)$ for $x \in \mathcal{X}$ and $v_i(y)$ for $y \in \mathcal{Y}$ are not partitions of unity because they can have negative values,
 and large negative values can be a source of numerical instability.
We can compute both $\mathbf{u}(z)$ and $\mathbf{v}(z)$ to high relative accuracy with $O(r)$ operations
 by using the modified Lagrange formula \cite[eq.\@ (3.1)]{barycentric_interpolation_stability} and precomputing all $z$-independent terms.
Thus the numerical errors from evaluating $\mathbf{u}(x)$ and $\mathbf{v}(y)$ in \cref{interpolant_matrices} are negligible
 relative to the numerical errors in the matrix products.

To compare the numerical errors and low-rank approximation errors in skeleton decompositions directly,
 we consider pointwise relative error bounds on the numerical errors that are compatible with \cref{minimax}.
For the left-sided and right-sided interpolative matrix decompositions in \cref{interpolant_matrices},
 these numerical error bounds are respectively
\begin{align}
 \left| fl \left( \mathbf{u}(x)^T \mathbf{C}(\tilde{\mathbf{x}},y) \right) - \mathbf{f}(x)^T \mathbf{g}(y) \right| &\le \epsilon_r \frac{\kappa_r(\mathcal{X},\mathcal{Y})}{|x - y|} \ \ \ \ \mathrm{and} \\
 \left| fl \left( \mathbf{C}(x,\tilde{\mathbf{y}}) \mathbf{v}(y) \right) - \mathbf{f}(x)^T \mathbf{g}(y) \right| &\le \epsilon_r \frac{\kappa_r(\mathcal{Y},\mathcal{X})}{|x - y|}, \ \ \ \ 
 (x,y) \in \mathcal{X} \times \mathcal{Y}, \notag
\end{align}
 for a relative condition number $\kappa_r(\mathcal{X}, \mathcal{Y})$ that satisfies \cref{sum_error} for all $x \in \mathcal{X}$ and $y \in \mathcal{Y}$.
We define a convenient but suboptimal relative condition number to be
\begin{equation} \label{skeleton_condition}
 \kappa_r(\mathcal{X}, \mathcal{Y}) \coloneqq \max_{x \in \mathcal{X}} \sum_{i=1}^r \max_{y \in \mathcal{Y}}\left| u_i(x)\frac{x - y}{\tilde{x}_i - y} \right|,
\end{equation}
 where $\tilde{\mathbf{x}}$ and $\tilde{\mathbf{y}}$ are the minimizers of \cref{simple_minimax}.
While it is possible to decrease $\kappa_r(\mathcal{X}, \mathcal{Y})$ by moving the maximization over $y$ outside of the summation,
 this increases the cost of computing $\kappa_r(\mathcal{X}, \mathcal{Y})$ and makes it unsuitable for bounding numerical errors in the two-sided decomposition.
With this choice of $\kappa_r(\mathcal{X}, \mathcal{Y})$, the two-sided error bound is
\begin{align}
 \left| fl \left( \mathbf{u}(x)^T \mathbf{C}(\tilde{\mathbf{x}},\tilde{\mathbf{y}}) \mathbf{v}(y) \right) - \mathbf{f}(x)^T \mathbf{g}(y) \right| &\le \epsilon_r \frac{\kappa_r(\mathcal{X},\mathcal{Y}) + \kappa_r(\mathcal{Y},\mathcal{X})}{|x - y|} \notag \\
 & \ \ \ \, + \epsilon_r^2 \frac{\kappa_r(\mathcal{X},\mathcal{Y}) \kappa_r(\mathcal{Y},\mathcal{X})}{|x - y|}, \ \ \ \ (x,y) \in \mathcal{X} \times \mathcal{Y}, \notag
\end{align}
 assuming that the floating-point computations are decomposed into an intermediate matrix-vector product followed by a vector inner product.
The numerical stability of \cref{interpolant_matrices} thus requires that $\kappa_r(\mathcal{X}, \mathcal{Y})$ and $\kappa_r(\mathcal{Y}, \mathcal{X})$ grow slowly with increasing $r$.

As a numerical example, we consider $\kappa_r(\lambda) \coloneqq \kappa_r([\lambda,1],[-1,-\lambda])$ corresponding to the analytical solutions in \cref{analytical_solution}.
As with $Z_r(\mathcal{X}, \mathcal{Y})$, $\kappa_r(\mathcal{X}, \mathcal{Y})$ is invariant to M\"{o}bius transformations of $\mathcal{X}$ and $\mathcal{Y}$,
 thus $\kappa_r([x_{\min},x_{\max}], [y_{\min}, y_{\max}]) = \kappa_r(\lambda)$ for $\lambda$ in \cref{mobius_transform}.
In the $\lambda \rightarrow 1$ limit of \cref{zolotarev_solution} shown in \cref{map_limits}, $\tilde{\mathbf{x}}$ and $\tilde{\mathbf{y}}$ approach Chebyshev nodes
 that are shifted and scaled from $[-1,1]$ to $[\lambda,1]$ and $[-1,-\lambda]$.
The weights and prefactors from \cref{skeleton_condition} vanish in this limit,
 reducing $\kappa_r(\lambda)$ to the Lebesgue constant for Chebyshev nodes,
 which has known bounds and asymptotes \cite{lebesgue_chebyshev}.
In the left panel of \cref{fig_condition},
 we plot the difference between $\kappa_r(\lambda)$ and its large-$r$ asymptote at $\lambda = 1$,
\begin{equation} \label{lebesgue_asymptote}
 \overline{\kappa}_r \coloneqq \frac{2}{\pi} \left( \gamma + \log \frac{8}{\pi} + \log r \right) ,
\end{equation}
 where $\gamma \approx 0.5772$ is the Euler--Mascheroni constant, and we observe this asymptotic $r$ dependence to persist for all $\lambda \in (0,1)$.
The $\lambda$-dependent offset of the asymptote,
\begin{equation} \label{asymptotic_condition}
  \hat{\kappa}(\lambda) \coloneqq \lim_{r \rightarrow \infty} \left( \kappa_r(\lambda) - \overline{\kappa}_r \right),
\end{equation}
 is plotted in the right panel of \cref{fig_condition}.
We can fit all available data for $\hat{\kappa}(\lambda)$
 to an absolute accuracy of $0.01$ with a rational approximant in the variable $\log \lambda$,
\begin{equation}
  \hat{\kappa}(\lambda) \approx \frac{ 0.305 (\log \lambda)^2}{5.88 - \log \lambda}. \notag
\end{equation}
This numerical example provides an empirical understanding of $\kappa_r(\lambda)$,
 but a rigorous understanding comparable to \cref{analytical_solution} will require substantially more work.

\begin{figure}[!t]
  \centering
  \label{fig_condition}\includegraphics{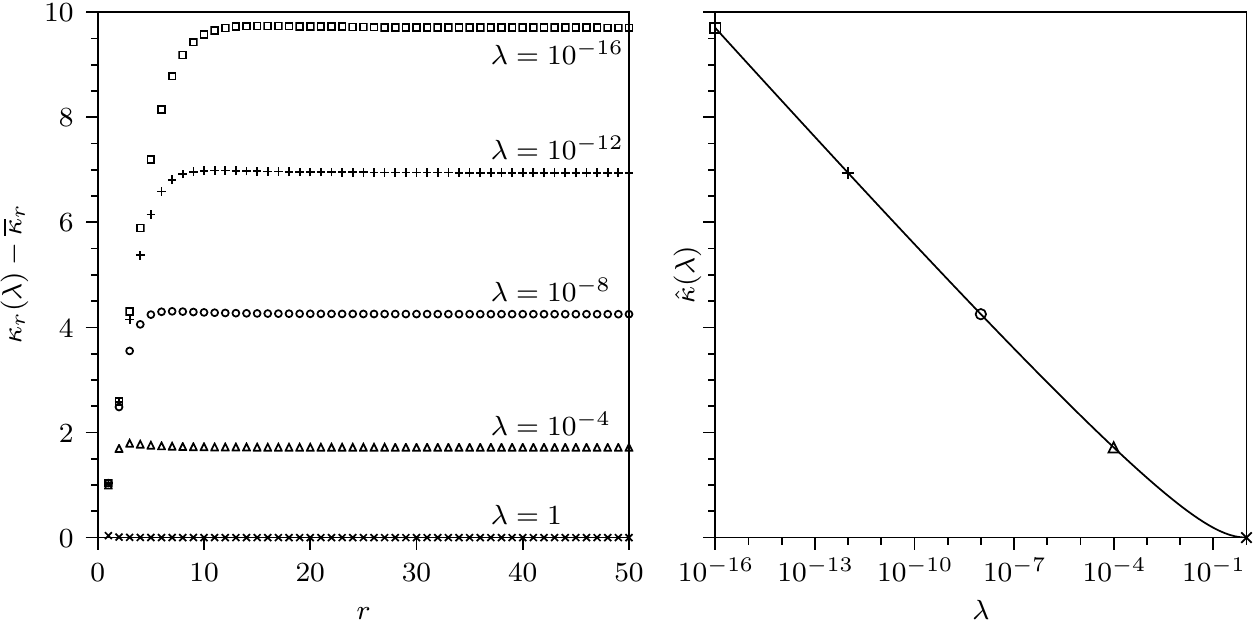}
  \caption{The relative condition number of a skeleton decomposition, $\kappa_r(\lambda) \coloneqq \kappa_r([\lambda,1],[-1,-\lambda])$ from \cref{skeleton_condition},
  offset by its large-$r$ asymptote at $\lambda = 1$, $\overline{\kappa}_r$ from \cref{lebesgue_asymptote}, (left panel)
  and the $\lambda$ dependence of the observed large-$r$ asymptote of this difference, $\hat{\kappa}(\lambda)$ from \cref{asymptotic_condition} (right panel). }
\end{figure}

We note that $\mathbf{u}(x)$ and $\mathbf{v}(y)$ in \cref{interpolant_matrices} each form a basis for rational interpolation
 and $\kappa_r(\mathcal{X},\mathcal{Y})$ in \cref{skeleton_condition} is related to their Lebesgue constants,
 which is an active topic of research \cite{rational_lebesgue, rational_lebesgue2}.
Unfortunately, the available theoretical results on this topic are not immediately applicable here.
Also, an important property of $Z_r(\mathcal{X}, \mathcal{Y})$,
\begin{equation}
 Z_r(\mathcal{X}' , \mathcal{Y}') \le Z_r(\mathcal{X}, \mathcal{Y}), \ \ \ \ \mathcal{X}' \subseteq \mathcal{X}, \ \ \ \ \mathcal{Y}' \subseteq \mathcal{Y}, \notag
\end{equation}
 does not apply to $\kappa_r(\mathcal{X}, \mathcal{Y})$, which makes it difficult to bound the value of $\kappa_r(\mathcal{X}, \mathcal{Y})$ without computing it.
However, it is straightforward to compute $\kappa_r(\mathcal{X}, \mathcal{Y})$ alongside $Z_r(\mathcal{X}, \mathcal{Y})$
 with no substantial increase in computational cost,
 and this is available in our software implementation of the numerical solver in \cref{numerical_solution}.

\subsection{Equivalence of minimized norms\label{svd_analysis}}

Because the equivalence of norms is only guaranteed for finite-dimensional spaces,
 this discussion is limited to the matrix SVD rather than the more general operator SVD.
Likewise, we consider the Cauchy matrix $\mathbf{C}(\mathbf{x}, \mathbf{y})$ for $\mathbf{x} \in \mathbb{R}^m$ and $\mathbf{y} \in \mathbb{R}^n$ such that $\min_i x_i > \max_i y_i$.
The inequalities for the equivalence
 between the 2-norm and elementwise relative maximum norm are
\begin{align} \label{norm_equivalence}
 \eta_- \| \mathbf{A} \|_{\mathbf{x}, \mathbf{y}} & \le \|  \mathbf{A} \|_2 \le \eta_+ \| \mathbf{A} \|_{\mathbf{x}, \mathbf{y}} , \ \ \ \
 \eta_+^{-1} \|  \mathbf{A} \|_2 \le \| \mathbf{A} \|_{\mathbf{x}, \mathbf{y}} \le \eta_-^{-1} \| \mathbf{A} \|_2, \\
   & \| \mathbf{A} \|_{\mathbf{x}, \mathbf{y}} \coloneqq \max_{i,j} | (x_i - y_j) [\mathbf{A}]_{i,j} | , \ \ \ \ \mathbf{A} \in \mathbb{R}^{m \times n}, \notag
\end{align}
 with coefficients that quantify the saturation of an inequality,
\begin{align}
 \eta_- &\coloneqq \min_{\mathbf{B} \in \mathbb{R}^{m \times n} } \frac{ \| \mathbf{B} \|_2}{\| \mathbf{B} \|_{\mathbf{x}, \mathbf{y}} } = \frac{1}{\max_{i,j} | x_i - y_j| }, \\
 \eta_+ &\coloneqq \max_{\mathbf{B} \in \mathbb{R}^{m \times n} } \frac{ \| \mathbf{B} \|_2}{\| \mathbf{B} \|_{\mathbf{x}, \mathbf{y}} } = \| \mathbf{C}(\mathbf{x},\mathbf{y}) \|_2 . \notag
\end{align}
The maximum is attained by $\mathbf{B} = \mathbf{C}(\mathbf{x},\mathbf{y})$,
 and the minimum is attained by $\mathbf{B}$ with one nonzero matrix element at the same location as a matrix element of $\mathbf{C}(\mathbf{x},\mathbf{y})$ with the smallest magnitude.
We prove these optimizers by constructing attained bounds,
\begin{align} \label{equivalence_coefficients}
 \eta_- &\ge \frac{1}{\max_{i,j} | x_i - y_j| } \min_{\mathbf{B} \in \mathbb{R}^{m \times n} } \frac{\|\mathbf{B}\|_2}{\|\mathbf{B}\|_{\max}} \ge \frac{1}{\max_{i,j} | x_i - y_j| } , \\
 \eta_+ &= \max_{\mathbf{E} \in \mathbb{R}^{m \times n}} \frac{\| \mathbf{C}(\mathbf{x},\mathbf{y}) \circ \mathbf{E}\|_2 }{\| \mathbf{E} \|_{\max}} \le
 \max_{\substack{\mathbf{p} \in \mathbb{R}^m \\ \mathbf{q} \in \mathbb{R}^n}} \sum_{i,j} \frac{ | p_i q_j [\mathbf{C}(\mathbf{x},\mathbf{y})]_{i,j} |}{ \| \mathbf{p} \|_2 \| \mathbf{q} \|_2 } \le \| \mathbf{C}(\mathbf{x}, \mathbf{y}) \|_2 , \notag
\end{align}
 where the lower bound on $\eta_-$ results from independently maximizing the two terms in $\| \mathbf{B} \|_{\mathbf{x}, \mathbf{y}}$
 and identifying that the elementwise maximum norm, $\|\mathbf{B}\|_{\max}$, is a lower bound for the 2-norm,
 and the upper bound results from changing matrix variables, $\mathbf{B} = \mathbf{C}(\mathbf{x},\mathbf{y}) \circ \mathbf{E}$ where $\circ$ is the elementwise matrix product,
 to split the variational form of the 2-norm with the H\"{o}lder inequality, $| \mathbf{a}^T \mathbf{b} | \le \| \mathbf{a} \|_1 \| \mathbf{b} \|_{\infty}$,
 and reform it by relaxing the elementwise sign constraints on $\mathbf{p}$ and $\mathbf{q}$ in the maximand.

Using \cref{norm_equivalence}, we now consider the transferability of skeleton decompositions and truncated SVDs.
We refer to the rank-$r$ skeleton decomposition that minimizes \cref{minimax}
 for $\mathcal{X} = \{ x_1, \cdots , x_m \}$ and $\mathcal{Y} = \{ y_1 , \cdots , y_n \}$ as $\widetilde{\mathbf{C}}_{r}^{\mathrm{skel}}(\mathbf{x},\mathbf{y})$
 and to the truncated SVD retaining the $r$ largest singular values and vectors to minimize \cref{svd_optimal} for $\mathbf{K} = \mathbf{C}(\mathbf{x},\mathbf{y})$
 as $\widetilde{\mathbf{C}}_{r}^{\mathrm{SVD}}(\mathbf{x},\mathbf{y})$.
We then define their transferability between norms as
\begin{align} \label{transferability}
 \mu_r^{\mathrm{skel}} &\coloneqq \frac{\| \mathbf{C}(\mathbf{x},\mathbf{y}) - \widetilde{\mathbf{C}}_{r}^{\mathrm{skel}}(\mathbf{x},\mathbf{y}) \|_2}
 {\| \mathbf{C}(\mathbf{x},\mathbf{y}) - \widetilde{\mathbf{C}}_{r}^{\mathrm{SVD}}(\mathbf{x},\mathbf{y}) \|_2}, \\
 \mu_r^{\mathrm{SVD}} &\coloneqq \frac{\| \mathbf{C}(\mathbf{x},\mathbf{y}) - \widetilde{\mathbf{C}}_{r}^{\mathrm{SVD}}(\mathbf{x},\mathbf{y}) \|_{\mathbf{x}, \mathbf{y}}}
 {\| \mathbf{C}(\mathbf{x},\mathbf{y}) - \widetilde{\mathbf{C}}_{r}^{\mathrm{skel}}(\mathbf{x},\mathbf{y}) \|_{\mathbf{x}, \mathbf{y}}}, \notag
\end{align}
 which are ratios between suboptimal and optimal approximation errors.
We use \cref{norm_equivalence} to construct upper bounds on $\mu_r^{\mathrm{skel}}$ and $\mu_r^{\mathrm{SVD}}$ alongside their trivial lower bounds,
\begin{equation} \label{transfer_bounds}
 1 \le \mu_r^{\mathrm{skel}} \le \eta_+/\eta_-, \ \ \ \ 1 \le \mu_r^{\mathrm{SVD}} \le \eta_+/\eta_-, \ \ \ \ \mu_r^{\mathrm{skel}} \mu_r^{\mathrm{SVD}} \le \eta_+/\eta_-.
\end{equation}
If $\eta_+ / \eta_- \approx 1$, then $\mu_r^{\mathrm{skel}} \approx \mu_r^{\mathrm{SVD}} \approx 1$ is guaranteed.
Otherwise, we cannot discount the possibility of a large $\mu_r^{\mathrm{skel}}$ or $\mu_r^{\mathrm{SVD}}$
 saturating a large upper bound in \cref{transfer_bounds}.

The near-saturation of the upper bound on $\mu_r^{\mathrm{skel}} \mu_r^{\mathrm{SVD}}$ in \cref{transfer_bounds} can be observed from numerical examples.
Here we consider $m = n$ with $\mathbf{x}$ and $\mathbf{y}$ assigned to be the local extrema of an analytical solution in \cref{zolotarev_solution} parameterized by $\lambda$.
We observe that $\| \mathbf{C}(\mathbf{x},\mathbf{y}) \|_2$ is within an order of magnitude of saturating the upper bound set by its elementwise maximum norm,
 which sets an upper bound on its transferability of
\begin{equation}
 \eta_+/\eta_- \le n/ \lambda 
\end{equation}
 that is similarly close to being saturated.
The denominator of $\mu_r^{\mathrm{SVD}}$ in \cref{transferability} can be bounded from above by $Z_r(\lambda)$,
 which is exact when $n-1$ is divisible by $r$ and tight for large $n$.
From \cref{Z_bounds}, $Z_r(\lambda)$ has an asymptotic exponential decay in $r$ with upper and lower bounds that converge.
An upper bound on transferability implies that this exponential decay must be inherited by other optimal low-rank approximations and the norms that they optimize,
 which has been proven for the truncated SVD and the 2-norm \cite[Cor.\@ 4.2]{displacement.rank.review}.
In \cref{fig_equivalence}, the transferability of both the truncated SVD and skeleton decomposition inevitably saturates at large $r$
 as their relative error increases and approaches the transferability bound.
In this saturated regime, the exponential decay of error is transferred to the suboptimal norm.
We observe that the truncated SVD and skeleton decomposition are equally transferable at large $r$,
 although we are unable to explain why transferability is so balanced between these approximants.

\begin{figure}[!t]
  \centering
  \label{fig_equivalence}\includegraphics{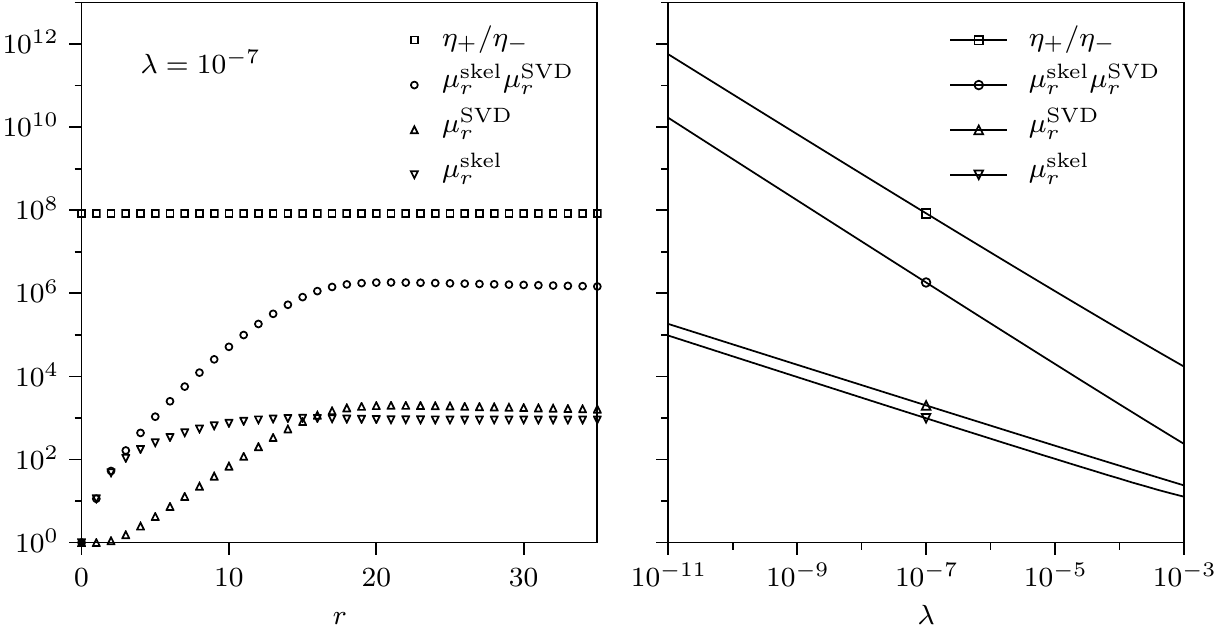}
  \caption{The transferability of the skeleton decomposition and truncated SVD,  $\mu_r^{\mathrm{skel}}$ and $\mu_r^{\mathrm{SVD}}$ from \cref{transferability},
  and the upper bound on their product, $\eta_+/\eta_-$ from \cref{equivalence_coefficients},
  for $\mathbf{C}(\mathbf{x},\mathbf{y})$ with $\mathbf{x}$ and $\mathbf{y}$ from \cref{zolotarev_solution} at $n = 99$,
 including their $r$ dependence at $\lambda = 10^{-7}$ (left panel) and the $\lambda$ dependence of their maximum value over $r$ (right panel).
The observed power laws in $\lambda$ of their maximum values are $\mu_r^{\mathrm{skel}} \propto \lambda^{-0.49}$, $\mu_r^{\mathrm{SVD}} \propto \lambda^{-0.49}$, and $\eta_+/\eta_- \propto \lambda^{-0.94}$. }
\end{figure}

To improve the transferability of optimal low-rank approximations between two norms,
 we can consider modifying a norm with diagonal matrices, $\mathbf{P} \coloneqq \mathrm{diag}(\mathbf{p})$ and $\mathbf{Q} \coloneqq \mathrm{diag}(\mathbf{q})$,
 which induces a weighted norm equivalence relation,
\begin{align}
 \eta_-(\mathbf{P},\mathbf{Q}) \| \mathbf{A} \|_{\mathbf{x}, \mathbf{y}} & \le \|  \mathbf{P} \mathbf{A} \mathbf{Q} \|_2 \le \eta_+(\mathbf{P},\mathbf{Q}) \| \mathbf{A} \|_{\mathbf{x}, \mathbf{y}} , \notag \\
 \eta_+(\mathbf{P}^{-1},\mathbf{Q}^{-1})^{-1} \|  \mathbf{A} \|_2 & \le \| \mathbf{P} \mathbf{A} \mathbf{Q} \|_{\mathbf{x}, \mathbf{y}} \le \eta_-(\mathbf{P}^{-1},\mathbf{Q}^{-1})^{-1} \| \mathbf{A} \|_2 . \notag
\end{align}
A truncated SVD of $\mathbf{P} \mathbf{A} \mathbf{Q}$ reweighted by $\mathbf{P}^{-1}$ on the left and $\mathbf{Q}^{-1}$ on the right is the
 optimal low-rank approximation relative to this weighted 2-norm.
Similarly, we can extend \cref{minimax_theorem} and \cref{Z_lemma} to include separable weight functions
 and produce a weighted skeleton decomposition that is the optimal low-rank approximation of the weighted elementwise relative maximum norm.
Thus, we can improve transferability while mostly preserving the familiar forms of low-rank approximation by choosing $\mathbf{P}$ and $\mathbf{Q}$
 to reduce $\eta_+(\mathbf{P},\mathbf{Q})/\eta_-(\mathbf{P},\mathbf{Q})$.
The coefficients of this norm equivalence are
\begin{align}
 \eta_-(\mathbf{P},\mathbf{Q}) &\coloneqq \min_{\mathbf{B} \in \mathbb{R}^{m \times n} } \frac{ \| \mathbf{P} \mathbf{B} \mathbf{Q} \|_2}{\| \mathbf{B} \|_{\mathbf{x}, \mathbf{y}} } = \min_{i,j} \left| \frac{p_i q_j}{x_i - y_j} \right|, \notag \\
 \eta_+(\mathbf{P},\mathbf{Q}) &\coloneqq \max_{\mathbf{B} \in \mathbb{R}^{m \times n} } \frac{ \| \mathbf{P} \mathbf{B} \mathbf{Q} \|_2}{\| \mathbf{B} \|_{\mathbf{x}, \mathbf{y}} } = \| \mathbf{P} \mathbf{C}(\mathbf{x},\mathbf{y}) \mathbf{Q} \|_2 , \notag
\end{align}
 with the same proof as \cref{equivalence_coefficients}.
For the example in \cref{fig_equivalence}, we can use the weights
\begin{equation}
 p_i = | x_i | + \sqrt{ \lambda }, \ \ \ \ q_i = | y_i | + \sqrt{\lambda}, \notag
\end{equation}
 to reduce an upper bound on $\eta_+(\mathbf{P},\mathbf{Q})/\eta_-(\mathbf{P},\mathbf{Q})$ set by the equivalence
 between the 2-norm and the elementwise maximum norm by a factor of $\approx \sqrt{\lambda}$,
\begin{equation}
 \frac{\eta_+(\mathbf{P},\mathbf{Q})}{\eta_-(\mathbf{P},\mathbf{Q})} \le n \frac{\max_{i,j}|p_i q_j / (x_i - y_j)|}{\min_{i,j}|p_i q_j / (x_i - y_j)|} = \frac{n(1 + \lambda)}{2 \sqrt{\lambda}},
\end{equation}
While these weights reduce the upper bound substantially from $\propto \lambda^{-1}$ to $\propto \lambda^{-1/2}$,
 transferability between weighted approximants remains poor for small $\lambda$.

\section{Conclusions\label{conclusions}}

Skeleton decompositions were originally proposed as heuristic alternatives to truncated SVDs
 for low-rank matrix approximations with a small but suboptimal 2-norm error \cite{skeleton}.
The main result of this paper, \cref{minimax_theorem}, has proven that skeleton decompositions
 have their own optimality result, specific to the Cauchy kernel and the maximum relative pointwise error.
It relates Zolotarev's work \cite{zolotarev.review,minimax.rational} on optimal rational approximation of functions to optimal low-rank approximation of matrices and operators.
The special property of the Cauchy kernel that enables this optimality result is the equivalence between
 its skeleton decompositions and rational interpolants shown in \cref{interpolant_matrices}.
Previous work \cite{displacement.rank.review,Hilbert_upper_bound} had proven $Z_r(\mathcal{X}, \mathcal{Y})$ as an upper bound in \cref{minimax},
 but the lower bound and its proof is a new result of this work.

There are several ways in which the results of this paper might be extended and expanded.
Although \cref{minimax_theorem} does not extend to complex-valued $\mathcal{X}$, $\mathcal{Y}$, $\tilde{\mathbf{x}}$, and $\tilde{\mathbf{y}}$,
 $Z_r(\mathcal{X}, \mathcal{Y})$ remains an upper bound on \cref{minimax}.
The minimizing roots and poles of $h(z)$ are not always simple in the complex case,
 as occurs in a known complex analytical solution \cite{complex_disk}.
\cref{minimax_theorem} and \cref{Z_lemma} can be extended to include positive separable weight functions $w(x) w(y)$ in their optimands.
Some steps in the proof of \cref{minimax_theorem} can be adapted to other kernel functions,
 but it is not clear if they can be leveraged into a useful result.
Numerical and heuristic solutions of \cref{minimax} were demonstrated in \cref{many_solutions},
 but more effective algorithms to construct numerical solutions and a more diverse set of heuristic solutions would be useful.
Proofs for the asymptotic values of Lebesgue constants for Chebyshev nodes \cite{lebesgue_chebyshev} 
 might be adapted to \cref{asymptotic_condition} by extending their use of
 trigonometric identities to the corresponding Jacobi elliptic functions.

\cref{minimax_theorem} has several immediate applications to numerical linear algebra.
First, the hierarchical factorization of real Cauchy matrices with high relative elementwise accuracy is possible
 by recursively partitioning a Cauchy matrix $\mathbf{C}(\mathbf{x},\mathbf{y})$ as
\begin{align}
 \mathbf{C}(\mathbf{x},\mathbf{y}) = \left[ \begin{array}{cc} \mathbf{C}(\mathbf{x}_1,\mathbf{y}_1) & \mathbf{C}(\mathbf{x}_1,\mathbf{y}_2) \\ \mathbf{C}(\mathbf{x}_2,\mathbf{y}_1) & \mathbf{C}(\mathbf{x}_2,\mathbf{y}_2) \end{array} \right], \notag
\end{align}
 where the elements of $\mathbf{x}$ and $\mathbf{y}$ ordered and partitioned such that
 \cref{minimax_theorem} can be applied to the off-diagonal matrix blocks as the process is recursed with the diagonal matrix blocks.
Such hierarchical factorizations might extend to other matrices with low displacement rank
 through their rank-preserving connection to Cauchy matrices \cite{displacement.rank.review}.
Second, techniques for the dimensional reduction of sparse symmetric eigenvalue problems
 use the Cauchy kernel as a component of spectral filtering \cite{rational_AMLS}.
The tight upper bound on pointwise relative error in \cref{minimax_theorem} can improve upon the efficacy of spectral filtering
 for use as a reliable primitive in future eigenvalue solvers.

\cref{minimax_theorem} is also useful in fast algorithms for the many-electron problem.
The energy denominators that occur in many-body perturbation theory can be separated
 using Cauchy kernels \cite{Moussa_RPA}, although they are usually separated with exponential sums \cite{Laplace_chemistry,Laplace_minimax}.
The maximum errors in these exponential sums have upper bounds that are proportional to Zolotarev numbers \cite{Laplace_bounds}
 but suboptimal relative to \cref{minimax_theorem}.
Using the optimality results of this paper as a guide, these bounds may be tightened, and asymptotically optimal limits may be identified.
Some fast algorithms for mean-field theory calculations \cite{PEXSI} use rational function approximations
 to relate general matrix functions to shifted matrix inverses.
While specific function approximations can be optimized \cite{rational_minimax},
 many of these functions share a common approximation domain and pole domain.
Skeleton decompositions of $1/(x-y)$ for $x \in \mathcal{X}$ and $y \in \mathcal{Y}$ can be used as a common function approximant
 if $\mathcal{X}$ is the approximation domain and $\mathcal{Y}$ is the pole domain.
Mastery of these approximation schemes can benefit fast algorithms by
 reducing computational cost prefactors and tightening computable error bounds.

\bibliographystyle{siamplain}
\bibliography{cauchy}
\end{document}